\documentclass[10pt,reqno]{amsart}

\usepackage[margin=1in]{geometry} 
\usepackage{amsmath,amsthm,amssymb,bbm,bm}
\usepackage{graphicx}
\usepackage{hyperref}
\usepackage{amssymb,amsfonts}
\usepackage{enumerate}
\usepackage{enumitem}
\usepackage{mathrsfs}
\usepackage{xcolor}
\usepackage{tikz-cd}

\newcommand{\R}{\mathbb{R}}

\newcommand{\E}{\mathbb{E}}

\newcommand{\Cs}{\mathscr{C}}

\newcommand{\Var}{\mathrm{Var}}
\newcommand{\cov}{\mathrm{cov}}

\newcommand{\Beta}{\mathtt{Beta}}

\newcommand{\Gam}{\mathtt{Gamma}}

\newcommand{\PD}{\mathtt{PD}}
\newcommand{\PYP}{\mathtt{PYP}}

\newcommand{\Address}{{
\bigskip
\footnotesize

\textsc{Department of Mathematics \& Statistics, McMaster University, Hamilton, ON, L8S 4K1, Canada}\par\nopagebreak
\textit{E-mail address}: \texttt{paguyoj@mcmaster.ca}
}}

\def\bal#1\eal{\begin{align*}#1\end{align*}}

\newtheorem{theorem}{Theorem}[section]
\newtheorem{lemma}[theorem]{Lemma}
\newtheorem{proposition}[theorem]{Proposition}

\title[CLT for the homozygosity of the HPYP]{Central limit theorem for the homozygosity of the hierarchical Pitman-Yor process}
\author{Shui Feng and J. E. Paguyo}
\date{}


\begin{document}

\begin{abstract}
The hierarchical Pitman-Yor process is a discrete random measure used as a prior in Bayesian nonparametrics. It is motivated by the study of groups of clustered data exhibiting power law behavior. Our focus in this paper is on the Gaussian behavior of a family of statistics, namely the power sum symmetric polynomials for the vector of weights of the process, as the concentration parameters tend to infinity. We establish a central limit theorem and obtain explicit representations for the asymptotic variance, with the latter clearly showing the impact of each component in the hierarchical structure. These results are crucial for understanding the asymptotic behavior of the sampling formulas associated with the process. In comparison with the known results for the hierarchical Dirichlet process, the results for the hierarchical Pitman-Yor process are mathematically more challenging and structurally more revealing of power law behavior. 
\end{abstract}

\maketitle


\section{Introduction} \label{Sec: Introduction}

Let 
\bal
\Delta  = \left\{\bm{x}= (x_1, x_2, \ldots): 0\leq x_i\leq 1, \text{for $ i=1,2,\ldots$, and } \sum_{i=1}^\infty x_i =1\right\}.
\eal
For $\bm{x} \in \Delta$, denote the {\em power sum symmetric polynomials} by
\bal
\varphi_m(\bm{x})=\sum_{i=1}^\infty x_i^m, \ \  m=2, 3, \ldots. 
\eal
Consider a random sample of size $m \geq 2$ from a population of individuals of various types, with type distribution given by $\bm{x} \in \Delta$. Then $\varphi_m(\bm{x})$ represents the probability that all individuals in the sample are of the same type. This quantity arises in many contexts, such as population genetics, ecology \cite{Sim49}, and economics \cite{Her50, Hir45}. In Bayesian statistics, $\bm{x}$ is replaced by a random vector. 

For $\alpha \in (0,1)$ and $\theta > -\alpha$, let $\{U_k\}_{k=1}^\infty$ be a sequence of independent random variables such that $U_k$ is $\Beta(1-\alpha, \theta + k\alpha)$ distributed. Let
\bal
V_1 = U_1 \quad \text{and} \quad V_n = \prod_{k=1}^{n-1} (1 - U_k)U_n, \quad \text{for $n \geq 2$},
\eal
and let $\bm{V} = \bm{V}(\alpha,\theta) = (V_1, V_2, \ldots)$. The {\em two parameter Poisson-Dirichlet distribution} with concentration parameter $\theta$ and discount parameter $\alpha$, denoted $\PD(\alpha, \theta)$, is the distribution of $\bm{V}$ in descending order. 

Independently of $\bm{V}$, let $S$ be a Polish space and let $M_1(S)$ denote the space of probabilities on $S$ equipped with the weak topology. For any $\nu$ in $M_1(S)$, let $\{\xi_k\}_{k=1}^\infty$ be independent and identically distributed (iid) random variables with common distribution $\nu$. The {\em Pitman-Yor process} with concentration parameter $\theta$, discount parameter $\alpha$, and base distribution $\nu$, denoted by $\PYP(\alpha, \theta, \nu)$, is the discrete random probability measure
\bal
\Xi_{\alpha, \theta, \nu} = \sum_{i=1}^\infty V_i \delta_{\xi_i}. 
\eal
A comprehensive discussion of the  two-parameter Poisson-Dirichlet distribution and the Pitman-Yor process is found in Pitman and Yor \cite{PY97}.  The case $\alpha = 0$ is the {\em Dirichlet process}, first introduced by Ferguson \cite{Fer73}.  Asymptotic behavior of the Poisson-Dirichlet distribution and the Pitman-Yor process has been the focus of intensive research in the past few decades \cite{Fen10, GV17}. 

The Pitman-Yor and Dirichlet processes are used as fundamental priors in Bayesian nonparametrics for the study of clustered data and serve as building blocks for priors in the study of more complex data. Let $\alpha, \beta \in [0,1)$, $\theta_0 > -\alpha$, and $\theta > -\beta$. The {\em hierarchical Pitman-Yor process} (henceforth HPYP) \cite{Teh06} restricted to one group with  concentration parameters $\theta_0, \theta$, discount parameters $\alpha, \beta$,  and base distribution $\nu$ is a Pitman-Yor process whose base distribution is a draw from another Pitman-Yor process,
\bal
\Xi_{\alpha, \theta_0, \beta, \theta, \nu} \stackrel{D}{=} \Xi_{\beta, \theta, \Xi_{\alpha, \theta_0, \nu}},
\eal
where $\stackrel{D}{=}$ denotes equality in distribution. Equivalently, it can be written as the discrete random probability measure
\bal
\Xi_{\alpha, \theta_0, \beta, \theta, \nu} = \sum_{i=1}^\infty Z_i \delta_{\xi_i}. 
\eal
where $\bm{Z} = \bm{Z}(\alpha, \theta_0, \beta, \theta) = (Z_1, Z_2, \ldots)$ is the vector of weights of the HPYP, which we call the {\em HPYP masses}. A subordinator representation of these masses will be derived and used to establish the main result. Setting $\alpha = \beta = 0$ yields the hierarchical Dirichlet process (henceforth HDP) \cite{TJBB06}.

Given an integer $L \geq 1$, the HPYP with $L$ groups, $(\Xi^1_{\alpha, \theta_0, \beta, \theta, \nu}, \ldots, \Xi^L_{\alpha, \theta_0, \beta, \theta, \nu})$, is a collection of discrete random measures that share the same Pitman-Yor process base measure $\Xi_{\alpha, \theta_0, \nu}$ and are conditionally iid given the base measure. We refer to the base Pitman-Yor process, $\Xi_{\alpha, \theta_0, \nu}$, as the {\em level one Pitman-Yor process}, and $\Xi^k_{\alpha, \theta_0, \beta, \theta, \nu}$, for $1\leq k \leq L$, as the {\em level two Pitman-Yor processes}. The level one Pitman-Yor process describes the global impact, while the level two Pitman-Yor processes describes the group specific impact. Accordingly, $\theta$  and $\beta$ are the level two concentration parameter and discount parameter, while $\theta_0$ and $\alpha$ are the level one concentration parameter and discount parameter.
For more on the hierarchical Dirichlet process and other hierarchical Bayesian nonparametric models, see \cite{ACV20, CLOP19, Teh06, TJ10, TJBB06} and the references therein. 

Consider an HPYP with L groups, $(\Xi^1_{\alpha, \theta_0, \beta, \theta, \nu}, \ldots, \Xi^L_{\alpha, \theta_0, \beta, \theta, \nu})$. For all $1 \leq k \leq L$, let ${\bf Z}_k=(Z_{k, 1}, Z_{k,2}, \ldots)$ be the HPYP masses of $\Xi^k_{\alpha, \theta_0, \beta, \theta, \nu}$, so that 
\bal
\Xi^k_{\alpha, \theta_0, \beta, \theta, \nu}=\sum_{i=1}^\infty Z_{k,i}\delta_{\xi_i}.
\eal
Define the set $M_{n, L} := \{\bm{n} = (n_1,\ldots, n_L) : \text{$n_k \geq 0$ for all $1 \leq k \leq L$ and $n_1 + \dotsb + n_L = n$}\}$.
The {\em $m$th order homozygosity} of the Pitman-Yor process and the HPYP with $L$ groups are defined, respectively, as
\bal
H_m(\alpha, \theta_0) &=\varphi_m(\bm{V}) = \sum_{i=1}^\infty V_i^m \quad \text{and} \quad H_{m,L}(\alpha, \theta_0, \beta, \theta) = \sum_{i=1}^\infty \sum_{\bm{m} \in M_{m,L}} \frac{1}{L^m} \prod_{k=1}^L Z_{k,i}^{m_k}.
\eal
Our definition for the homozygosity of the HPYP with $L$ groups is for a random sample of size $m$, where each sample is drawn equally likely from any one of the $L$ groups. 

The origin of the term homozygosity comes from population genetics, where it corresponds specifically to $H_2(0,\theta)$. It is a statistic used to test neutrality of a population.  We will use the term loosely in this paper. There has been recent work on the asymptotic behavior of the homozygosity as the concentration parameter tends to infinity. Central limit theorems (henceforth CLT) for $H_m(\alpha, \theta_0)$ have been established in \cite{Gri79, JKK02, JKK03, Han09, Fen10}, for both the $\alpha = 0$ and $\alpha \in (0,1)$ cases. Large deviation principles were established in \cite{DF06, DF16} and moderate deviation principles in \cite{FG08, FG10}. 

There have been new developments in our understanding of hierarchical Bayesian nonparametric models in past few years.  These include the comprehensive study of the distribution theory in \cite{CLOP19, ACV20} and the large deviation principle for HDP \cite{Fen23}. More recently, the authors established the CLT and large deviation principle for the homozygosity of the HDP and the finite dimensional HDP \cite{FP25}. Asymptotic behavior of the number of clusters for the HPYP are investigated in \cite{FFP25}.

The collection of the homozygosities of different orders are special statistics of the random sample. They play a crucial role in understanding the sampling distribution and in deriving the sampling formula. In this paper, we consider the asymptotic behavior of the homozygosity of the HPYP as $\theta_0, \theta$ both tend to infinity in the case where $\alpha, \beta \in (0,1)$. In comparison with the corresponding result for HDP,  the main difficulties we face are the loss of the Gamma-Dirichlet algebra and the extra layer of randomness. 
In population genetics, large values of concentration parameters  correspond to a locus consisting of a large number of sites, with a fixed mutation rate at each site \cite{Gri79}.  In the context of Bayesian inference, the large concentration limits are associated with the Pitman-Yor posterior limits for large sample size. In ecology, this limiting regime corresponds to a system where there are a large number of species with small proportions. 


\subsection{Notation}

Let $a_n, b_n$ be two sequences. If $\lim_{n \to \infty} \frac{a_n}{b_n}$ is a nonzero constant, then we write $a_n \asymp b_n$. If the limit is $1$, we write $a_n\sim b_n$. If there exists positive constants $c$ and $n_0$ such that $a_n \leq cb_n$ for all $n \geq n_0$, then we write $a_n = O(b_n)$. 

Convergence in distribution, in probability, and almost surely are denoted by $\stackrel{D}{\longrightarrow}$, $\stackrel{P}{\longrightarrow}$, and $\stackrel{a.s.}{\longrightarrow}$, respectively. Similarly, equality in distribution, in probability, and almost surely are denoted by $\stackrel{D}{=}$, $\stackrel{P}{=}$, and $\stackrel{a.s.}{=}$, respectively. If $X$ is a random variable distributed as $\nu$, then we write $X \sim \nu$. 

Let $(x)_{(n)} = x(x+1)(x+2)\dotsb (x+n-1)$ be the {\em rising factorial}. The {\em generalized factorial coefficients} \cite{Cha05} are defined as
\bal
\Cs(m,j,\beta) = \frac{1}{j!} \sum_{k=0}^j (-1)^k \binom{j}{k}(-k\beta)_{(m)}. 
\eal

Let ${n \brack k}$ denote the {\em unsigned Stirling numbers of the first kind}, which are defined to be the coefficients of the rising factorial $(x)_{(n)} = \sum_{\ell = 0}^n {n \brack \ell} x^\ell$. The number ${n \brack k}$ is also defined as the number of permutations on $n$ elements with exactly $k$ cycles. Let ${n \brace k}$ be a {\em Stirling number of the second kind}, which is the number of ways to partition a set of size $n$ into $k$ nonempty subsets. 

\subsection{Main Results}

Our main result is the CLT for $H_{m,L}(\alpha, \theta_0, \beta, \theta)$, the homozygosity of the HPYP with $L$ groups.

\begin{theorem} \label{HPYPCLTLGroups}
Under the limiting procedure $\theta_0, \theta \to \infty$ such that $\frac{\theta_0}{\theta} \to c \in (0,\infty)$,
the homozygosity of the HPYP with $L$ groups satisfies
\bal
\tilde{H}_{m,L}(\alpha, \theta_0, \beta, \theta) &= \frac{\sqrt{\theta}}{f(\alpha, \beta, \theta; m, L, c)} \left( H_{m,L}(\alpha, \theta_0, \beta, \theta) - \frac{\sum_{j=1}^m A_j(\beta, m, L) \frac{ \prod_{s=1}^{j-1} (\theta + s\beta) (1 - \alpha)_{(j-1)} }{ \beta^j \theta_0^{j-1}} }{L^m \theta^{m-1}} \right) \\
&\xrightarrow{D} N(0, \sigma_{c,m,L}^2)
\eal
for all $m \geq 2$, where the variance is given by
\bal
&\sigma_{c,m,L}^2 = \sigma_{X,m,L}^2 + \sigma_{T,m,L}^2 + \sigma_{1,m,L}^2 + \sum_{k=1}^L \left( \sum_{\bm{m} \in M_{m,L}} C_{\bm{m}} m_k \right)^2 \\
&- 2\sum_{k=1}^L \left( \frac{ \sum_{j=1}^m \left[\left(\sum_{\bm{m} \in M_{m,L}} \sum_{\bm{j} \in M_{j, L}} (m_k - \beta j_k) \prod_{\ell = 1}^L \Cs(m_\ell, j_\ell, \beta) \right) + \beta j A_j(\beta, m, L) \right] \frac{(1 - \alpha)_{(j-1)}}{\beta^j c^{j-1}} }{\sum_{j=1}^m A_j(\beta, m, L) \frac{(1 - \alpha)_{(j-1)}}{\beta^j c^{j-1}}} \right)^2
\eal
where
\bal
\sigma_{1,m,L}^2 &= \frac{\sum_{1 \leq i,j \leq m} \frac{A_i(\beta, m, L)A_j(\beta, m, L)\left[ (1 - \alpha)_{(i+j-1)} + (\alpha - ij)(1 - \alpha)_{(i-1)}(1 - \alpha)_{(j-1)} \right]}{\beta^{i+j} c^{i+j-1}}}{\left( \sum_{j=1}^m A_j(\beta, m, L) \frac{(1 - \alpha)_{(j-1)}}{\beta^j c^{j-1}} \right)^2} \\
\sigma_{T,m,L}^2 &= \beta \frac{ \sum_{1 \leq i,j \leq m} \frac{ A_i(\beta, m, L)A_j(\beta, m, L) (1 - \alpha)_{(i-1)}(1 - \alpha)_{(j-1)} }{\beta^{i+j} c^{i+j-2}} \cdot ij}{\left( \sum_{j=1}^m A_j(\beta, m, L) \frac{(1 - \alpha)_{(j-1)}}{\beta^j c^{j-1}} \right)^2} \\
\sigma_{X,m,L}^2 &= \frac{\sum_{j=1}^{2m} \tilde{A}_j(\beta, m, L) \frac{(1 - \alpha)_{(j-1)}}{\beta^j c^{j-1}} -  \sum_{1 \leq i,j \leq m} A_i(\beta, m, L) A_j(\beta, m, L)  \frac{(1 - \alpha)_{(i+j-1)}}{\beta^{i+j} c^{i+j-1}}  }{\left( \sum_{j=1}^m A_j(\beta, m, L) \frac{(1 - \alpha)_{(j-1)}}{\beta^j c^{j-1}} \right)^2}
\eal
and the coefficients $\{A_j(\beta, m, L)\}_{1 \leq j \leq m}$, $\{\tilde{A}_j(\beta, m, L)\}_{1 \leq j \leq 2m}$, and $\{C_{\bm{m}}\}_{\bm{m} \in M_{m,L}}$ are defined, respectively, as
\bal
A_j(\beta, m, L) &= \sum_{\bm{m} \in M_{m,L}} \sum_{\bm{j} \in M_{j,L}} \prod_{k=1}^L \Cs(m_k,j_k,\beta) \\
\tilde{A}_j(\beta, m, L) &= \sum_{\bm{m}_1 \in M_{m,L}} \sum_{\bm{m}_2 \in M_{m,L}} \sum_{\bm{j} \in M_{j,L}} \prod_{k=1}^L \Cs(m_{1k} + m_{2k}, j_k, \beta) \\
C_{\bm{m}} &= \frac{\sum_{j=1}^m \sum_{\bm{j} \in M_{j,L}} \prod_{\ell = 1}^L \Cs(m_\ell, j_\ell, \beta) \frac{(1 - \alpha)_{(j-1)}}{\beta^j c^{j-1}}}{\sum_{j=1}^m A_j(\beta, m, L) \frac{(1 - \alpha)_{(j-1)}}{\beta^j c^{j-1}}}.
\eal
\end{theorem}

\noindent {\bf Remarks:}
\begin{enumerate}
\item Consider the single group case $L = 1$. When $c$ tends to infinity, the level 1 concentration parameter, $\theta_0$, goes to infinity faster than the level 2 concentration parameter, $\theta$, and the Gaussian fluctuation is solely determined by the level two Pitman-Yor process. Using the fact that $\Cs(m,1,\beta) = \beta(1-\beta)_{(m-1)}$, we have that
\bal
&\lim_{c \to \infty} \sigma_{c,m,1}^2 \\
&= \lim_{c \to \infty} \frac{\sum_{j=1}^{2m} \frac{\Cs(2m, j, \beta) (1 - \alpha)_{(j-1)}}{\beta^j c^{j-1}} + \sum_{1 \leq i,j \leq m} \frac{\Cs(m,i,\beta) \Cs(m,j,\beta) (\alpha - ij + \beta c ij) (1 - \alpha)_{(i-1)} (1 - \alpha)_{(j-1)}}{\beta^{i+j} c^{i+j-1}} }{\left( \sum_{j=1}^m \frac{ \Cs(m,j,\beta) (1 - \alpha)_{(j-1)} }{ \beta^j c^{j-1}} \right)^2} - m^2 \\
&= \frac{(1-\beta)_{(2m-1)}}{(1-\beta)_{(m-1)}^2} + \beta - m^2,
\eal
which recovers the single level homozygosity CLT from \cite{Han09}; see Lemma \ref{PYPCLT} below. 
\item By Lemma \ref{limittoStirling}, $\lim_{\beta \to 0} \frac{\Cs(m,j,\beta)}{\beta^j} = {m \brack j}$, an unsigned Stirling number of the first kind. Thus setting $\alpha = 0$ and taking $\beta \to 0$ in Theorem \ref{HPYPCLTLGroups} recovers the CLT for the homozygosity of the HDP with $L$ groups (Theorem 1.4, \cite{FP25}).
\end{enumerate}

\subsection{Outline}

The remaining of the paper is organized as follows. In Section \ref{Sec: Preliminaries}, we introduce auxiliary results and technical lemmas that we use throughout the paper. In Section \ref{Sec: CLTLGroups} we prove Theorem \ref{HPYPCLTLGroups}. 


\section{Preliminaries} \label{Sec: Preliminaries}

\subsection{Auxiliary Results}

Handa's CLT for the homozygosity of the Pitman-Yor process \cite{Han09} will be used throughout the paper. The version we state is from Feng \cite{Fen10}.

\begin{lemma}[Theorem 7.11, \cite{Fen10}] \label{PYPCLT}
Suppose $\bm{X} = (X_1, X_2, \ldots) \sim \PD(\alpha, \theta)$. Let 
\bal
\tilde{H}_m(\alpha, \theta) = \sqrt{\theta}\left( \frac{\theta^{m-1}}{(1 - \alpha)_{(m-1)}}  \sum_{i = 1}^\infty X_i^m - 1 \right), \qquad m \geq 2,
\eal
be the scaled $m$th order homozygosity. Then
\bal
(\tilde{H}_2(\alpha, \theta), \tilde{H}_3(\alpha, \theta), \ldots) \xrightarrow{D} (H_2(\alpha), H_3(\alpha), \ldots)
\eal
as $\theta \to \infty$, where $\bm{H} = (H_2(\alpha), H_3(\alpha), \ldots)$ is a $\R^\infty$-valued random vector such that for all $r \geq 2$, $(H_2(\alpha), \ldots, H_r(\alpha))$ has a multivariate normal distribution with zero mean and covariance matrix
\bal
\cov(H_i(\alpha), H_j(\alpha)) = \frac{(1 - \alpha)_{(i+j-1)}}{(1 - \alpha)_{(i-1)}(1 - \alpha)_{(j-1)}} + \alpha - ij, \qquad i,j = 2, \ldots, r.
\eal
\end{lemma}

The following combinatorial lemma was used in the introduction to show that the CLT for the homozygosity of the HDP (Theorem 1.4, \cite{FP25}) follows as a special case of Theorem \ref{HPYPCLTLGroups} .

\begin{lemma} \label{limittoStirling}
Let $0 \leq j \leq m$. Then
\bal
\frac{\Cs(m,j,\beta)}{\beta^j} \to {m \brack j}
\eal
as $\beta \to 0$. 
\end{lemma}

\begin{proof}
Recalling that $(x)_{(n)} = \sum_{\ell = 0}^n {n \brack \ell} x^\ell$, we can rewrite
\bal
\Cs(m,j,\beta) &= \frac{1}{j!} \sum_{k=0}^j (-1)^k \binom{j}{k}(-k\beta)_{(m)} = \frac{1}{j!} \sum_{k=0}^j (-1)^k \binom{j}{k} \sum_{\ell = 0}^m {m \brack \ell} (-k\beta)^\ell \\
&= \sum_{\ell = 0}^m {m \brack \ell} (-1)^\ell \beta^\ell \left( \frac{1}{j!} \sum_{k=0}^j (-1)^k \binom{j}{k} k^\ell \right).
\eal 
Using the formula for the Stirling numbers of the second kind, ${n \brace k} = \frac{1}{k!} \sum_{j=0}^k (-1)^{k-j} \binom{k}{j} j^n$, we can write
\bal
\frac{1}{j!} \sum_{k=0}^j (-1)^k \binom{j}{k} k^\ell = (-1)^{-j} {\ell \brace j}.
\eal
Therefore we get
\bal
\Cs(m,j,\beta) = \sum_{\ell = 0}^m (-1)^{\ell-j} {m \brack \ell} {\ell \brace j} \beta^\ell = {m \brack j}\beta^j + \sum_{\ell = j+1}^m(-1)^{\ell-j} {m \brack \ell} {\ell \brace j} \beta^\ell , 
\eal
from which it follows that 
\bal
\frac{\Cs(m,j,\beta)}{\beta^j} = {m \brack j} + \sum_{\ell = j+1}^m {m \brack j}\beta^j + \sum_{\ell = j+1}^m(-1)^{\ell-j} {m \brack \ell} {\ell \brace j} \beta^\ell \to {m \brack j}
\eal
as $\beta \to 0$. 
\end{proof}

\subsection{Subordinator Representation of HPYP} \label{Sec:Subordinatorrep}

Let $\Gamma(x)$ be the {\em gamma function}. Let $\alpha \in (0,1)$, $\theta_0 > -\alpha$, and $C > 0$. Let $\gamma(t)$ be the gamma subordinator with Levy density $\theta_0 x^{-1} e^{-x}$, and independently, let $\sigma(t)$ be the subordinator with Levy density $\alpha x^{-\alpha - 1} e^{-x}$. Define
\bal
\sigma_{\alpha, \theta_0}(t) = \sigma\left( \frac{ \gamma(1/\alpha) t}{C\Gamma(1 - \alpha)} \right).
\eal
and $\sigma_{\alpha, \theta_0} = \sigma_{\alpha, \theta_0}(1)$. Let $J_1(\alpha, \theta_0) \geq J_2(\alpha, \theta_0) \geq \dotsb$ be the ranked jump sizes of $\sigma(t)$ over the random interval $[0, (C\Gamma(1-\alpha))^{-1}\gamma(1/\alpha)]$. Then $\sigma_{\alpha, \theta}$ and $\left( \frac{J_1(\alpha, \theta)}{\sigma_{\alpha, \theta}}, \frac{J_2(\alpha, \theta)}{\sigma_{\alpha, \theta}}, \ldots \right)$
are independent, and have distributions $\Gam(\theta, 1)$ and $\PD(\alpha, \theta)$, respectively. 

The HPYP masses, $\bm{Z} = (Z_1, Z_2, \ldots)$, also admit a subordinator representation. Let $\alpha, \beta \in (0,1)$, $\theta_0 > -\alpha$, and $\theta > -\beta$. Given the level 1 masses, $\bm{V} = (V_1, V_2, \ldots)$, define the increments
\bal
W_k(\alpha, \theta_0, \beta, \theta) := \sigma_{\beta, \theta}\left( \sum_{i=1}^k V_i \right) - \sigma_{\beta, \theta}\left( \sum_{i=1}^{k-1} V_i \right) \stackrel{D}{=} \sigma_{\beta, \theta}\left( V_k \right),
\eal
for all $k \geq 1$. Then
\bal
\bm{Z} = (Z_1, Z_2, \ldots) \stackrel{D}{=} \left( \frac{W_1(\alpha, \theta_0, \beta, \theta)}{\sigma_{\beta,\theta}}, \frac{W_2(\alpha, \theta_0, \beta, \theta)}{\sigma_{\beta,\theta}}, \ldots \right).
\eal

The following lemma establishes the conditional moments of the increments $W_k^m(\alpha, \theta_0, \beta, \theta)$. 

\begin{lemma} \label{incrementsconditionalmean}
For all $m \geq 1$ and $k \geq 1$, 
\bal
\E[W_k^m(\alpha, \theta_0, \beta, \theta) \mid \bm{V}, \gamma_{1/\beta}] &= \sum_{j=1}^m \Cs(m,j,\beta) \gamma_{1/\beta}^j V_k^j, \\
\E[W_k^m(\alpha, \theta_0, \beta, \theta) \mid \bm{V}] &= \sum_{j=1}^m \frac{\Cs(m,j,\beta) \prod_{s=0}^{j-1} (\theta + s\beta)}{\beta^j} V_k^j,
\eal
where $\Cs(m,j,\beta)$ is the generalized factorial coefficient. 
\end{lemma}

\begin{proof}
By direct computation, the Laplace transforms of $\gamma(t)$ and $\sigma(t)$ are
\bal
\E\left[ e^{-\lambda \gamma(t)} \right] = \exp\left( -t\theta \int_0^\infty (1 - e^{-\lambda x}) x^{-1} e^{-x} \, dx \right) = (1 + \lambda)^{-\theta t}
\eal
and
\bal
\E\left[ e^{-\lambda \sigma(t)} \right] &= \exp\left( -t \alpha C \int_0^\infty (1 - e^{-\lambda x}) x^{-(1+\alpha)} e^{-x} \, dx \right) = \exp(-tC\Gamma(1-\alpha)[(\lambda + 1)^\alpha - 1]),
\eal
respectively. 
Thus the conditional Laplace transform of $\sigma_{\beta, \theta}(V_k)$ given $\bm{V}$ and $\gamma_{1/\beta}$ is
\bal
\E\left[ e^{-\lambda \sigma_{\beta, \theta}(V_k)} \middle\vert \bm{V}, \gamma_{1/\beta} \right] = \exp(-\gamma_{1/\beta}V_k[(\lambda+1)^\beta - 1]).
\eal
The conditional moments of $W_k^m(\alpha, \theta_0, \beta, \theta)$ can be computed as
\bal
\E[W_k^m(\alpha, \theta_0, \beta, \theta) \mid \bm{V}, \gamma_{1/\beta}] &= \E[\sigma_{\beta, \theta}(V_k)^m \mid \bm{V}, \gamma_{1/\beta}] = \left. (-1)^m \frac{d^m}{d\lambda^m} \E\left[ e^{-\lambda \sigma_{\beta, \theta}(V_k)} \middle\vert \bm{V}, \gamma_{1/\beta} \right] \right|_{\lambda = 0} \\
&= \left. (-1)^m \frac{d^m}{d\lambda^m} \exp(-\gamma_{1/\beta}V_k[(\lambda+1)^\beta - 1]) \right|_{\lambda = 0}.
\eal
Let $\exp(-\gamma_{1/\beta}V_k[(\lambda+1)^\beta - 1]) = f(u(\lambda))$ where $f(u) = e^u$ and $u(\lambda) = -\gamma_{1/\beta}V_k[(\lambda+1)^\beta - 1]$. Then by Faa di Bruno's formula, 
\bal
&(-1)^m \frac{d^m}{d\lambda^m} \exp(-\gamma_{1/\beta}V_k[(\lambda+1)^\beta - 1]) = (-1)^m \frac{d^m f(u(\lambda))}{d \lambda^m} \\
&\qquad = (-1)^m \sum_{ \sum_{i=1}^m ik_i = m } \frac{m!}{\prod_{i=1}^m k_i! i!^{k_i}} \cdot f^{(k_1 + \dotsb + k_m)}(u(\lambda)) \cdot \prod_{i=1}^m (u^{(i)}(\lambda))^{k_i}.
\eal
By direct computation,
\bal
\left. \frac{d^j f(u)}{du^j} \right|_{\lambda = 0} &= \left. e^u \right|_{\lambda = 0} = 1, \\
\left. \frac{d^j u(\lambda)}{d\lambda^j} \right|_{\lambda = 0} &= \left. -\gamma_{1/\beta} V_k \beta(\beta - 1) \dotsb (\beta - j + 1) (\lambda + 1)^{\beta - j} \right|_{\lambda = 0} = -\gamma_{1/\beta} V_k (\beta)_{j}
\eal
for all $1 \leq j \leq m$, where $(x)_n = x(x-1)(x-2)\dotsb (x-n+1)$ is the {\em falling factorial}. Combining the above gives 
\bal
\E[\sigma_{\beta, \theta}(V_k)^m \mid \bm{V}] &= \left. (-1)^m \sum_{ \sum_{i=1}^m ik_i = m } \frac{m!}{\prod_{i=1}^m k_i! i!^{k_i}} \cdot f^{(k_1 + \dotsb + k_m)}(u(\lambda)) \cdot \prod_{i=1}^m (u^{(i)}(\lambda))^{k_i} \right|_{\lambda = 0} \\
&= \sum_{j=1}^m \left[ \sum_{ \substack{ \sum_{i=1}^m ik_i = m \\ \sum_{i=1}^m k_i = j} } m! \left(\prod_{i=1}^m \frac{[(\beta)_i]^{k_i}}{k_i! i!^{k_i}}\right) \right] \cdot (-1)^{m+j} \gamma_{1/\beta}^j V_k^j \\
&= \sum_{j=1}^m (-1)^{m+j} \tilde{\Cs}(m,j,\beta) \gamma_{1/\beta}^j V_k^j,
\eal
where in the last equality we applied Theorems 2.14 and 2.15 in \cite{Cha05}, which combined states that
\bal
\tilde{\Cs}(m,j,\beta) := \frac{1}{j!} \sum_{k=0}^j (-1)^{j-k} \binom{j}{k} (\beta k)_{m} &= \sum_{ \substack{ \sum_{i=1}^m ik_i = m \\ \sum_{i=1}^m k_i = j} } \frac{m!}{\prod_{i=1}^m k_i!} \prod_{i=1}^m \binom{\beta}{i}^{k_i}.
\eal
Using the identity $(-k\beta)_{(m)} = (-1)^m (k\beta)_m$, which follows from the fact that the rising and falling factorials are related to one another through the identity $(x)_{(n)} = (-1)^n (-x)_{n}$, we get that
\bal
(-1)^{m+j} \tilde{\Cs}(m,j,\beta) = \frac{1}{j!} \sum_{k=0}^j (-1)^{m-k} \binom{j}{k} (\beta k)_{m} = \frac{1}{j!} \sum_{k=0}^j (-1)^k \binom{j}{k} (-k \beta)_{(m)} =: \Cs(m,j,\beta), 
\eal
from which it follows that 
\bal
\E[W_k^m(\alpha, \theta_0, \beta, \theta) \mid \bm{V}, \gamma_{1/\beta}] &= \sum_{j=1}^m \Cs(m,j,\beta) \gamma_{1/\beta}^j V_k^j.
\eal
Finally by direct computation and using the properties of gamma moments,
\bal
\E[W_k^m(\alpha, \theta_0, \beta, \theta) \mid \bm{V}] &= \sum_{j=1}^m \Cs(m,j,\beta) \E[\gamma_{1/\beta}^j] V_k^j = \sum_{j=1}^m \frac{\Cs(m,j,\beta) \prod_{s=0}^{j-1} (\theta + s\beta)}{\beta^j} V_k^j. \qedhere
\eal
\end{proof}

\subsection{Some Technical Lemmas}

In this section, we establish some technical lemmas that will be used in the sequel.

\begin{lemma} \label{level1weightmoments}
For all $k \geq 1$ and $j \geq 0$, 
\bal
\E[V_k^j] = \frac{(1 - \alpha)_{(j)}}{(1 + \theta)_{(j)}} \prod_{s = 1}^{k-1} \frac{\theta + s \alpha}{\theta + s\alpha + j}.
\eal
\end{lemma}

\begin{proof}
Using the fact that $\{U_k\}$ are independent random variables such that $U_k \sim \Beta(1 - \alpha, \theta + k\alpha)$, a direct computation yields
\bal
\E[V_k^j] &= \E\left[U_k^j \prod_{s=1}^{k-1} (1 - U_s)^j\right] = \prod_{r=0}^{j-1} \frac{1 - \alpha + r}{1 - \alpha + \theta + s\alpha + r} \prod_{s=1}^{k-1} \prod_{r=0}^{j-1} \frac{\theta + s\alpha + r}{1 - \alpha + \theta + s\alpha + r} \\
&= \frac{(1- \alpha)_{(j)}}{(1 + \theta + \alpha(k-1))_{(j)}} \prod_{s=1}^{k-1} \frac{(\theta + s\alpha)_{(j)}}{(1 + \theta + \alpha(s-1))_{(j)}} = \frac{(1- \alpha)_{(j)} \prod_{s=1}^{k-1} (\theta + s\alpha)_{(j)}}{\prod_{s=1}^{k} (1 + \theta + \alpha(s-1))_{(j)}} \\
&= \frac{(1 - \alpha)_{(j)}}{(1 + \theta)_{(j)}} \prod_{s=1}^{k-1} \frac{(\theta + s\alpha)_{(j)}}{(1 + \theta + s\alpha)_{(j)}} = \frac{(1 - \alpha)_{(j)}}{(1 + \theta)_{(j)}} \prod_{s = 1}^{k-1} \frac{\theta + s \alpha}{\theta + s\alpha + j}. \qedhere
\eal
\end{proof}

\begin{lemma} \label{level1weightspartialsum}
Let $k \geq 1$ and $j \geq 0$. Then for all $N \geq 1$,
\bal
\sum_{k=1}^N \E[V_k^j] = \frac{(1 - \alpha)_{(j-1)} \left(\theta + j - (\theta + N\alpha) \prod_{s = 1}^{N-1} \frac{\theta + s \alpha}{\theta + s\alpha + j} \right)}{(1 + \theta)_{(j)}}.
\eal
\end{lemma}

\begin{proof}
Let $A_k = \prod_{s = 1}^{k-1} \frac{\theta + s \alpha}{\theta + s\alpha + j}$ for all $k \geq 1$, so that 
\bal
A_{k+1} = A_k\left( \frac{\theta + k\alpha}{\theta + k\alpha + j} \right) \implies A_k = \frac{\theta + k\alpha + j}{j}(A_k - A_{k+1}). 
\eal
Summing both sides over $k$ gives
\bal
\sum_{k=1}^N A_k &= \frac{1}{j} \sum_{k=1}^N (\theta + k\alpha + j)(A_k - A_{k+1}) = \frac{1}{j}\left[(\theta + \alpha + j) + \alpha \sum_{k=2}^N A_k - (\theta + N\alpha + j) A_{N+1} \right] \\
&= \frac{\alpha}{j} \sum_{k=1}^N A_k + \frac{\theta + j - (\theta + N\alpha + j) A_{N+1} }{j},
\eal
and rearranging yields
\bal
\sum_{k=1}^N A_k = \frac{\theta + j - (\theta + N\alpha + j) A_{N+1}}{j-\alpha}. 
\eal
Finally by Lemma \ref{level1weightmoments}, 
\bal
\sum_{k=1}^N \E[V_k^j] &= \frac{(1 - \alpha)_{(j)}}{(1 + \theta)_{(j)}} \sum_{k=1}^N A_k = \frac{(1 - \alpha)_{(j)}}{(1 + \theta)_{(j)}} \cdot \frac{\theta + j - (\theta + N\alpha + j) A_{N+1}}{j-\alpha} \\
&= \frac{(1 - \alpha)_{(j-1)} \left(\theta + j - (\theta + N\alpha) \prod_{s = 1}^{N-1} \frac{\theta + s \alpha}{\theta + s\alpha + j} \right)}{(1 + \theta)_{(j)}}. \qedhere
\eal
\end{proof}

\begin{lemma}\label{asconvHomozyg}
For all $j \geq 1$, 
\bal
\frac{\theta^{j-1}}{(1 - \alpha)_{(j-1)}} H_j(\alpha, \theta) \xrightarrow{a.s.} 1
\eal
Moreover this almost sure convergence also holds if $H_j(\alpha, \theta)$ is replaced by $\sum_{k=1}^{\lfloor \theta^r \rfloor} V_k^j$, where $r$ is an integer such that $r = r(\alpha) = \lfloor \frac{1}{1-\alpha} \rfloor + 1$. 
\end{lemma}

\begin{proof}
First recall that 
\bal
\E[H_j(\alpha, \theta)] = \frac{(1 - \alpha)_{(j-1)}}{(1 + \theta)_{(j-1)}}.
\eal
Next we compute the second moment $\E[H_j^2(\alpha, \theta)]$. By direct computation, 
\bal
\E[V_k^{2j}] = \E\left[U_k^{2j} \prod_{s = 1}^{k-1} (1 - U_s)^{2j}\right] = \frac{(1 - \alpha)_{(2j)}}{(1 + \theta + \alpha(k-1))_{(2j)}} \prod_{s = 1}^{k-1} \frac{(\theta + s\alpha)_{(2j)}}{(1 + \theta + \alpha(s-1))_{(2j)}},
\eal
and for $k < l$,
\bal
\E[V_k^j V_l^j] &= \E\left[ U_k^j \prod_{s = 1}^{k-1} (1 - U_s)^j U_l^j \prod_{r = 1}^{l-1} (1 - U_r)^j \right] = \E\left[ U_k^j (1 - U_k)^j U_l^j \prod_{s = 1}^{k-1} (1 - U_s)^{2j} \prod_{s = k+1}^{l-1} (1 - U_s)^j \right] \\ 
&= \frac{(1 - \alpha)_{(j)} (\theta + k\alpha)_{(j)}}{(1 + \theta + \alpha(k-1))_{(2j)}} \frac{(1 - \alpha)_{(j)}}{(1 + \theta + \alpha(l-1))_{(j)}} \prod_{s = 1}^{k-1} \frac{(\theta + s\alpha)_{(2j)}}{(1 + \theta + \alpha(s-1))_{(2j)}} \prod_{s = k+1}^{l-1} \frac{(\theta + s\alpha)_{(j)}}{(1 + \theta + \alpha(s-1))_{(j)}}. 
\eal
Thus combining the above yields
\bal
\E[H_j^2(\alpha, \theta)] &= \sum_{k,l = 1}^\infty \E[V_k^j V_l^j] = \sum_{k=1}^\infty \E[V_k^{2j}] + 2\sum_{1 \leq k < l \leq \infty} \E[V_k^j V_l^j] \\
&= \sum_{k=1}^\infty \frac{(1 - \alpha)_{(2j)}}{(1 + \theta + \alpha(k-1))_{(2j)}} \prod_{s = 1}^{k-1} \frac{(\theta + s\alpha)_{(2j)}}{(1 + \theta + \alpha(s-1))_{(2j)}} \\
&\qquad \qquad + 2\sum_{1 \leq k < l \leq \infty} \left[ \frac{(1 - \alpha)_{(j)} (\theta + k\alpha)_{(j)}}{(1 + \theta + \alpha(k-1))_{(2j)}} \frac{(1 - \alpha)_{(j)}}{(1 + \theta + \alpha(l-1))_{(j)}} \right. \\
&\qquad \qquad \left. \times \prod_{s = 1}^{k-1} \frac{(\theta + s\alpha)_{(2j)}}{(1 + \theta + \alpha(s-1))_{(2j)}} \prod_{s = k+1}^{l-1} \frac{(\theta + s\alpha)_{(j)}}{(1 + \theta + \alpha(s-1))_{(j)}} \right].
\eal
By direct computation, 
\bal
\E\left[ \left( \frac{\theta^{j-1}}{(1 - \alpha)_{(j-1)}} H_j(\alpha, \theta) - 1 \right)^2 \right] &= \Var\left( \frac{\theta^{j-1}}{(1 - \alpha)_{(j-1)}} H_j(\alpha, \theta) \right) + \left( \E\left[ \frac{\theta^{j-1}}{(1 - \alpha)_{(j-1)}} H_j(\alpha, \theta) \right] - 1 \right)^2 \\
&= \frac{1}{(1 - \alpha)\theta} + O(\theta^{-2}) \asymp \theta^{-1},
\eal
as $\theta \to \infty$. For any $a > 1$, set $\theta_n = \lfloor a^n \rfloor$ so that $\theta_n \to \infty$ as $n \to \infty$. Then by the Borel-Cantelli lemma, $\frac{\theta_n^{j-1}}{(1 - \alpha)_{(j-1)}} H_j(\theta_n) \to 1$ almost surely as $n \to \infty$. 

For general $\theta$, we can find $n$ such that $\theta_n \leq \theta \leq \theta_{n+1}$, and the almost sure convergence of $H_j(\alpha, \theta)$ follows by the squeeze theorem.  

Finally the truncated version follows from the fact that $\frac{\theta^{j-1}}{(1 - \alpha)_{(j-1)}} \sum_{k= \lfloor \theta^r \rfloor + 1}^\infty V_k^j \to 0$ almost surely as $\theta \to \infty$. 
\end{proof}


\section{CLT for the homozygosity of the HPYP with $L$ groups} \label{Sec: CLTLGroups}

Recall that the homozygosity of the HPYP with $L$ groups is defined as
\bal
H_{m,L}(\alpha, \theta_0, \beta, \theta) &= \sum_{i=1}^\infty \sum_{\bm{m} \in M_{m,L}} \frac{1}{L^m} \prod_{\ell=1}^L Z_{\ell i}^{m_\ell},
\eal
where $\bm{Z}_\ell = (Z_{\ell 1}, Z_{\ell 2}, \ldots)$ are the HPYP weights of the $\ell$th group. In this section we prove Theorem \ref{HPYPCLTLGroups}, the CLT for $H_{m,L}(\alpha, \theta_0, \beta, \theta)$. 

\subsection{Subordinator representation, law of large numbers, and decomposition for the HPYP with $L$ groups}

For all $1 \leq \ell \leq L$, the HPYP weights of the $\ell$th group has the subordinator representation
\bal
\bm{Z}_\ell = (Z_{\ell 1}, Z_{\ell 2}, \ldots) = \left( \frac{W_{\ell 1}(\alpha, \theta_0, \beta, \theta)}{\sigma_{\beta, \theta, \ell}}, \frac{W_{\ell 2}(\alpha, \theta_0, \beta, \theta)}{\sigma_{\beta, \theta, \ell}},\ldots \right),
\eal
where the increments are defined by
\bal
W_{\ell k}(\alpha, \theta_0, \beta, \theta) = \sigma_{\beta, \theta, \ell}\left( \sum_{i=1}^k V_i \right) - \sigma_{\beta, \theta, \ell}\left( \sum_{i=1}^{k-1} V_i \right) \stackrel{D}{=} \sigma_{\beta, \theta, \ell}(V_k), 
\eal
with $\sigma_{\beta, \theta, \ell}(t) = \sigma_\ell\left( \frac{\gamma_{1/\beta}t}{C\Gamma(1 - \beta)} \right)$ and $\bm{V} = (V_1, V_2, \ldots)$ the level 1 masses. 
We start by computing the expected value of $H_{m,L}(\alpha, \theta_0, \beta, \theta)$. 

\begin{lemma} \label{homozygosityLgroupsmean}
For all $m \geq 2$, 
\bal
\E[H_{m,L}(\alpha, \theta_0, \beta, \theta)] = \frac{1}{L^m} \sum_{\bm{m} \in M_{m,L}} \frac{1}{ \prod_{\ell=1}^L (\theta)_{(m_\ell)} }  \sum_{j=1}^m \left( \sum_{\bm{j} \in M_{j,L}} \prod_{\ell=1}^L \Cs(m_\ell,j_\ell,\beta) \right) \frac{\prod_{s=0}^{j-1} (\theta + s\beta) (1 - \alpha)_{(j-1)} }{\beta^j (\theta_0 + 1)_{(j-1)}}.
\eal
\end{lemma}

\begin{proof}
Conditioning on $\bm{V}$ and $\gamma_{1/\beta}$ and using Lemma \ref{incrementsconditionalmean} yields the conditional expectation
\bal
&\E[H_{m,L}(\alpha, \theta_0, \beta, \theta) \mid \bm{V}, \gamma_{1/\beta}] = \frac{1}{L^m} \sum_{i=1}^\infty \sum_{\bm{m} \in M_{m,L}} \prod_{\ell=1}^L \left( \frac{1}{(\theta)_{(m_\ell)}} \sum_{j=1 \wedge m_\ell}^{m_\ell} \Cs(m_\ell,j,\beta) \gamma_{1/\beta}^j V_i^j  \right) \\
&= \frac{1}{L^m} \sum_{\bm{m} \in M_{m,L}} \frac{1}{ \prod_{\ell=1}^L (\theta)_{(m_\ell)} }  \sum_{j=1}^m \left( \sum_{\bm{j} \in M_{j,L}} \prod_{\ell=1}^L \Cs(m_\ell,j_\ell,\beta) \right) \gamma_{1/\beta}^j \sum_{i=1}^\infty V_i^j.
\eal
Applying the law of total expectation finishes the proof. 
\end{proof}

Define the coefficients $\{A_j(\beta, m, L)\}_{1 \leq j \leq m}$ as
\bal
A_j(\beta, m, L) = \sum_{\bm{m} \in M_{m,L}} \sum_{\bm{j} \in M_{j,L}} \prod_{\ell=1}^L \Cs(m_\ell,j_\ell,\beta). 
\eal
By Lemma \ref{homozygosityLgroupsmean}, observe that the mean of the homozygosity is asymptotically
\bal
\E[H_{m,L}(\alpha, \theta_0, \beta, \theta)] &\sim \frac{1}{L^m \theta^{m-1}} \sum_{j=1}^m A_j(\beta, m, L) \frac{(1 - \alpha)_{(j-1)}}{\beta^j c^{j-1}} =: f(\alpha, \beta, \theta; m, L, c)
\eal
as $\theta_0, \theta \to \infty$ such that $\frac{\theta_0}{\theta} \to c$. Define
\bal
\tilde{f}(\alpha, \beta; m, L, c) = \sum_{j=1}^m A_j(\beta, m, L) \frac{(1 - \alpha)_{(j-1)}}{\beta^j c^{j-1}},
\eal
so that $L^m \theta^{m-1} f(\alpha, \beta, \theta; m, L, c) = \tilde{f}(\alpha, \beta; m, L, c)$. 

The next result is the law of large numbers for $H_{m,L}(\alpha, \theta_0, \beta, \theta)$. The proof follows by Chebyshev's inequality.

\begin{proposition} \label{homozygosityLgroupsLLN}
The homozygosity satisfies
\bal
\frac{H_{m,L}(\alpha, \theta_0, \beta, \theta)}{f(\alpha, \beta, \theta; m, L, c)} \xrightarrow{P} 1
\eal
as $\theta_0, \theta \to \infty$ such that $\frac{\theta_0}{\theta} \to c$.
\end{proposition}

The law of large numbers suggest a scaling of the form
\bal
\tilde{H}_{m,L}(\alpha, \theta_0, \beta, \theta) = \frac{\sqrt{\theta}\left( H_{m,L}(\alpha, \theta_0, \beta, \theta) - \frac{1}{L^m \theta^{m-1}} \sum_{j=1}^m A_j(\beta, m, L) \frac{ \prod_{s=1}^{j-1} (\theta + s\beta) (1 - \alpha)_{(j-1)}}{\beta^j \theta_0^{j-1}} \right)}{f(\alpha, \beta, \theta; m, L, c)},
\eal
which we refer to as the {\em scaled $m$th order homozygosity}. 

For notational simplicity, in what follows we suppress notation and write 
\bal
\tilde{f} = \tilde{f}(\alpha, \beta; m, L, c) \quad \text{and} \quad W_{\ell i} = W_{\ell i}(\alpha, \theta_0, \beta, \theta).
\eal
The scaled $m$th order homozygosity can be decomposed as
\bal
\tilde{H}_{m,L}(\alpha, \theta_0, \beta, \theta) &= X_{m,L}(\alpha, \theta_0, \beta, \theta) + T_{m,L}(\alpha, \theta_0, \beta, \theta) + Y_{m,L}(\alpha, \theta_0, \beta, \theta) \\
&\qquad + \sum_{\bm{m} \in M_{m,L}} \sqrt{\theta} \left( \prod_{\ell=1}^L \left( \frac{\theta}{\sigma_{\beta, \theta, \ell}} \right)^{m_\ell} - 1 \right) \frac{\sum_{i=1}^\infty \prod_{\ell=1}^L W_{\ell i}^{m_\ell}}{\theta \tilde{f}},
\eal
where the {\em level 1}, {\em gamma}, and {\em stable contributions} are defined, respectively, as
\bal
X_{m,L}(\alpha, \theta_0, \beta, \theta) &= \frac{1}{\sqrt{\theta} \tilde{f}} \sum_{i=1}^\infty \sum_{\bm{m} \in M_{m,L}} \left( \prod_{\ell=1}^L W_{\ell i}^{m_\ell} - \E\left[ \prod_{\ell=1}^L W_{\ell i}^{m_\ell} \middle\vert \bm{V}, \gamma_{1/\beta}\right] \right)  \\
T_{m,L}(\alpha, \theta_0, \beta, \theta) &= \frac{1}{\sqrt{\theta} \tilde{f}} \sum_{i=1}^\infty \sum_{\bm{m} \in M_{m,L}} \left( \E\left[ \prod_{\ell=1}^L W_{\ell i}^{m_\ell} \middle\vert \bm{V}, \gamma_{1/\beta}\right] - \E\left[ \prod_{\ell=1}^L W_{\ell i}^{m_\ell} \middle\vert \bm{V} \right] \right) \\
Y_{m,L}(\alpha, \theta_0, \beta, \theta) &= \frac{1}{\sqrt{\theta} \tilde{f}} \left( \sum_{i=1}^\infty \sum_{\bm{m} \in M_{m,L}} \E\left[ \prod_{\ell=1}^L W_{\ell i}^{m_\ell} \middle\vert \bm{V} \right] \right. \\
&\qquad \qquad \qquad \qquad \left. - \sum_{j=1}^m A_j(\beta, m, L) \frac{\prod_{s=1}^{j-1} (\theta + s\beta)(1 - \alpha)_{(j-1)}}{\beta^j \theta_0^{j-1}} \right),
\eal

We record the following useful lemma, which follows by Lemma \ref{incrementsconditionalmean} and direct calculation. 
\begin{lemma} \label{extendedincrementsmean}
For all $m \geq 1$ and $L \geq 1$, 
\bal
\sum_{i=1}^\infty \sum_{\bm{m} \in M_{m,L}} \E\left[ \prod_{\ell=1}^L W_{\ell i}^{m_\ell} \middle\vert \bm{V}, \gamma_{1/\beta} \right] &= \sum_{j=1}^m A_j(\beta, m, L) \gamma_{1/\beta}^j \sum_{i=1}^\infty V_i^j, \\
\sum_{i=1}^\infty \sum_{\bm{m} \in M_{m,L}} \E\left[ \prod_{\ell=1}^L W_{\ell i}^{m_\ell} \middle\vert \bm{V} \right] &= \sum_{j=1}^m A_j(\beta, m, L) \frac{\prod_{s=0}^{j-1} (\theta + s\beta)}{\beta^j} \sum_{i=1}^\infty V_i^j. 
\eal
\end{lemma}

\subsection{Level 1 contribution for the HPYP with $L$ groups}

\begin{lemma} \label{homozygosityLgroupsLevel1CLT}
The level 1 contribution satisfies
\bal
Y_{m,L}(\alpha, \theta_0, \beta, \theta) \xrightarrow{D} N(0, \sigma_{1,m,L}^2),
\eal
as $\theta_0, \theta \to \infty$ such that $\frac{\theta_0}{\theta} \to c$, where
\bal
\sigma_{1,m,L}^2 = \frac{\sum_{1 \leq i,j \leq m} \frac{A_i(\beta, m, L)A_j(\beta, m, L)\left[ (1 - \alpha)_{(i+j-1)} + (\alpha - ij)(1 - \alpha)_{(i-1)}(1 - \alpha)_{(j-1)} \right]}{\beta^{i+j} c^{i+j-1}}}{\left( \sum_{j=1}^m A_j(\beta, m, L) \frac{(1 - \alpha)_{(j-1)}}{\beta^j c^{j-1}} \right)^2}.
\eal
\end{lemma}

\begin{proof}
By Lemma \ref{extendedincrementsmean} and direct computation, we can rewrite the level 1 contribution as
\bal
Y_{m,L}(\alpha, \theta_0, \beta, \theta) = \frac{\sqrt{\theta/\theta_0}}{\tilde{f}} \sum_{j=2}^m A_j(\beta, m, L) \frac{\prod_{s=1}^{j-1} (\theta + s\beta)(1 - \alpha)_{(j-1)}}{\beta^j \theta_0^{j-1}} \tilde{H}_j(\alpha, \theta_0),
\eal
where $\tilde{H}_j(\alpha, \theta_0) = \sqrt{\theta_0}\left( \frac{\theta_0^{j-1}}{(1 - \alpha)_{(j-1)}}  \sum_{i = 1}^\infty X_i^j - 1 \right)$ is the scaled level 1 homozygosity. From this representation, it follows that $Y_{m,L}$ is asymptotically normal with mean $0$ and variance
\bal
\sigma_{1,m,L}^2 &= \frac{\sum_{1 \leq i,j \leq m} \frac{A_i(\beta, m, L)A_j(\beta, m, L)\left[ (1 - \alpha)_{(i+j-1)} + (\alpha - ij)(1 - \alpha)_{(i-1)}(1 - \alpha)_{(j-1)} \right]}{\beta^{i+j} c^{i+j-1}}}{\tilde{f}^2}. \qedhere
\eal
\end{proof}

\subsection{Gamma and stable contributions for the HPYP with $L$ groups}

\begin{lemma} \label{homozygosityLgroupsLevel2.1CLT}
The gamma contribution satisfies
\bal
T_{m,L}(\alpha, \theta_0, \beta, \theta) \xrightarrow{D} N(0, \sigma_{T,m,L}^2),
\eal
as $\theta_0, \theta \to \infty$ such that $\frac{\theta_0}{\theta} \to c$, where 
\bal
\sigma_{T,m,L}^2 &= \beta \frac{ \sum_{1 \leq i,j \leq m} \frac{ A_i(\beta, m, L)A_j(\beta, m, L) (1 - \alpha)_{(i-1)}(1 - \alpha)_{(j-1)} }{\beta^{i+j} c^{i+j-2}} \cdot ij}{\left( \sum_{j=1}^m A_j(\beta, m, L) \frac{(1 - \alpha)_{(j-1)}}{\beta^j c^{j-1}} \right)^2}.
\eal
\end{lemma}

\begin{proof}
By Lemma \ref{extendedincrementsmean} and direct computation, we can rewrite the gamma contribution as
\bal
&T_{m,L}(\alpha, \theta_0, \beta, \theta) =  \frac{1}{\sqrt{\theta} \tilde{f}} \sum_{j=1}^m A_j(\beta, m, L) \sum_{i=1}^\infty V_i^j \left( \gamma_{1/\beta}^j - \frac{\prod_{s=0}^{j-1} (\theta + s\beta)}{\beta^j} \right) \\
&= \frac{1}{\tilde{f}} \sum_{j=1}^m \frac{A_j(\beta, m, L) (1 - \alpha)_{(j-1)}}{\beta^j (\theta_0/\theta)^{j-1}} \left( \frac{\theta_0^{j-1}}{(1 - \alpha)_{(j-1)}} \sum_{i=1}^\infty V_i^j \right) \sqrt{\theta} \left( \left(\frac{\beta \gamma_{1/\beta}}{\theta}\right)^j - \frac{\prod_{s=0}^{j-1} (\theta + s\beta)}{\theta^j} \right),
\eal
which can then be decomposed into a sum of two terms as
\begin{align} \label{Level2.1Decomp}
\begin{split}
T_{m,L}(\alpha, \theta_0, \beta, \theta) &= \frac{\sqrt{\theta/\theta_0}}{\tilde{f}} \sum_{j=1}^m \frac{A_j(\beta, m, L) (1 - \alpha)_{(j-1)}}{\beta^j (\theta_0/\theta)^{j-1}} \tilde{H}_j(\alpha, \theta_0) \left( \left(\frac{\beta \gamma_{1/\beta}}{\theta}\right)^j - \frac{\prod_{s=0}^{j-1} (\theta + s\beta)}{\theta^j} \right) \\
&\qquad + \frac{1}{\tilde{f}} \sum_{j=1}^m \frac{A_j(\beta, m, L) (1 - \alpha)_{(j-1)}}{\beta^j (\theta_0/\theta)^{j-1}} \sqrt{\theta} \left( \left(\frac{\beta \gamma_{1/\beta}}{\theta}\right)^j - \frac{\prod_{s=0}^{j-1} (\theta + s\beta)}{\theta^j} \right).
\end{split}
\end{align}

For any $j \geq 1$, the law of large numbers for iid gamma random variables implies that
\bal
\left( \left(\frac{\beta \gamma_{1/\beta}}{\theta}\right)^j - \frac{\prod_{s=0}^{j-1} (\theta + s\beta)}{\theta^j} \right) \xrightarrow{P} 0
\eal
as $\theta \to \infty$. Thus by Slutsky's theorem and Lemmas \ref{PYPCLT} and \ref{asconvHomozyg}, the first summand of (\ref{Level2.1Decomp}) converges to $0$ in probability as $\theta_0, \theta \to \infty$. 

Next, observe that we can write
\bal
\sqrt{\theta} \left( \left(\frac{\beta \gamma_{1/\beta}}{\theta}\right)^j - \frac{\prod_{s=0}^{j-1} (\theta + s\beta)}{\theta^j} \right) 
&= \sqrt{\theta} \left( \left(\frac{\beta \gamma_{1/\beta}}{\theta}\right)^j - 1\right) + \sqrt{\theta}\left(1 - \frac{\prod_{s=0}^{j-1} (\theta + s\beta)}{\theta^j} \right). 
\eal
The second summand above converges to $0$ as $\theta \to \infty$. On the other hand, the CLT for a sum of iid Gamma random variables gives
\bal
\sqrt{\theta}\left( \frac{\beta \gamma_{1/\beta}}{\theta} - 1 \right) \xrightarrow{D} N(0,\beta).
\eal
Define the function $g(x) = (x, x^2, \ldots, x^m)$ so that $\nabla g(x) = (1, 2x, \ldots, mx^{m-1})$. By the multivariate Delta method, 
\begin{align} \label{gammamultivariateCLT}
\begin{split}
\left(\sqrt{\theta} \left( \left(\frac{\beta \gamma_{1/\beta}}{\theta}\right) - 1\right), \sqrt{\theta} \left( \left(\frac{\beta \gamma_{1/\beta}}{\theta}\right)^2 - 1\right), \ldots, \sqrt{\theta} \left( \left(\frac{\beta \gamma_{1/\beta}}{\theta}\right)^m - 1\right) \right) &= \sqrt{\theta}\left( g\left(\frac{\beta \gamma_{1/\beta}}{\theta}\right) - g(1) \right) \\
&\xrightarrow{D} N(\bm{0}, \bm{\Sigma}),
\end{split}
\end{align}
where the covariance matrix is given by
\begin{align} \label{GammaCovMatrix}
\begin{split}
\bm{\Sigma} = \nabla g(1)^T \cdot \beta \cdot \nabla g(1) = \beta \begin{pmatrix}
1 & 2 & 3 & \dotsb & m \\ 
2 & 4 & 6 & \dotsb & 2m \\
3 & 6 & 9 & \dotsb & 3m \\
\vdots & \vdots & \vdots & \ddots & \vdots \\
m & 2m & 3m & \dotsb & m^2
\end{pmatrix}.
\end{split}
\end{align}
Combining the above, it follows that the second summand of (\ref{Level2.1Decomp}) converges in distribution to $T \sim N(0, \sigma_{T,m,L}^2)$, where the asymptotic variance is given by
\bal
\sigma_{T,m,L}^2 &= \beta \frac{ \sum_{1 \leq i,j \leq m} \frac{ A_i(\beta, m, L)A_j(\beta, m, L) (1 - \alpha)_{(i-1)}(1 - \alpha)_{(i-1)} }{\beta^{i+j} c^{i+j-2}} \cdot ij}{ \tilde{f}^2 }. 
\eal
The proof follows by another application of Slutsky's theorem. 
\end{proof}

\begin{lemma} \label{homozygosityLgroupsLevel2.2CLT}
The stable contribution satisfies
\bal
X_{m,L}(\alpha, \theta_0, \beta, \theta) \xrightarrow{D} N(0, \sigma_{X,m,L}^2),
\eal
as $\theta_0, \theta \to \infty$ such that $\frac{\theta_0}{\theta} \to c$, where 
\bal
\sigma_{X,m,L}^2 &= \frac{\sum_{j=1}^{2m} \tilde{A}_j(\beta, m, L) \frac{(1 - \alpha)_{(j-1)}}{\beta^j c^{j-1}} -  \sum_{1 \leq i,j \leq m} A_i(\beta, m, L) A_j(\beta, m, L)  \frac{(1 - \alpha)_{(i+j-1)}}{\beta^{i+j} c^{i+j-1}}  }{\left( \sum_{j=1}^m A_j(\beta, m, L) \frac{(1 - \alpha)_{(j-1)}}{\beta^j c^{j-1}} \right)^2}
\eal
and the coefficients $\{\tilde{A}_j(\beta, m, L)\}_{1 \leq j \leq 2m}$ are defined by
\bal
\tilde{A}_j(\beta, m, L) = \sum_{\bm{m}_1 \in M_{m,L}} \sum_{\bm{m}_2 \in M_{m,L}} \sum_{\bm{j} \in M_{j,L}} \prod_{k=1}^L \Cs(m_{1k} + m_{2k}, j_k, \beta).
\eal
\end{lemma}

\begin{proof}
Let 
\bal
X_{m,L,1}(\alpha, \theta_0, \beta, \theta) &= \frac{1}{\sqrt{\theta} \tilde{f}} \sum_{k=1}^{\lfloor \theta_0^r \rfloor} \sum_{\bm{m} \in M_{m,L}} \left( \prod_{\ell=1}^L W_{\ell k}^{m_\ell} - \E\left[ \prod_{\ell=1}^L W_{\ell k}^{m_\ell} \middle\vert \bm{V}, \gamma_{1/\beta}\right] \right) \\
X_{m,L,2}(\alpha, \theta_0, \beta, \theta) &= \frac{1}{\sqrt{\theta} \tilde{f}} \sum_{k = \lfloor \theta_0^r \rfloor + 1}^\infty \sum_{\bm{m} \in M_{m,L}} \left( \prod_{\ell=1}^L W_{\ell k}^{m_\ell} - \E\left[ \prod_{\ell=1}^L W_{\ell k}^{m_\ell} \middle\vert \bm{V}, \gamma_{1/\beta}\right] \right)
\eal
where $r$ is an integer such that $r = r(\alpha) = \lfloor \frac{1}{1 - \alpha} \rfloor + 1$.

First we show that the tail, $X_{m,L,2}(\alpha, \theta_0, \beta, \theta)$, converges to $0$ as $\theta_0, \theta \to \infty$ such that $\frac{\theta_0}{\theta} \to c$. By Lemma \ref{level1weightspartialsum} and direct computation, 
\bal
\sum_{k= N + 1}^\infty \E[V_k^j] &= \sum_{k= 1}^\infty \E[V_k^j] - \sum_{k=1}^N \E[V_k^j] = \frac{(1 - \alpha)_{(j-1)}}{(1 +\theta_0)_{(j-1)}} - \frac{(1 - \alpha)_{(j-1)}}{(1 +\theta_0)_{(j-1)}}\left[ 1 - \frac{(\theta_0 + N\alpha)\prod_{s = 1}^{N-1} \frac{\theta_0 + s \alpha}{\theta_0 + s\alpha + j} }{\theta_0 + j} \right] \\
&= \frac{(1 - \alpha)_{(j-1)}}{(1 +\theta_0)_{(j-1)}}  \frac{(\theta_0 + N\alpha) \prod_{s = 1}^{N-1} \frac{\theta_0 + s \alpha}{\theta_0 + s\alpha + j} }{\theta_0 + j} = \frac{(1 - \alpha)_{(j-1)}}{(1 +\theta_0)_{(j-1)}} \frac{\theta_0 + N\alpha + j}{\theta_0 + j} \prod_{s = 1}^{N} \frac{\theta_0 + s \alpha}{\theta_0 + s\alpha + j}
\eal
for all $N \geq 1$. In particular, setting $N = \lfloor \theta_0^r \rfloor$ gives
\bal
\sum_{k= \lfloor \theta_0^r \rfloor + 1}^\infty \E[V_k^j] = \frac{(1 - \alpha)_{(j-1)}}{(1 +\theta_0)_{(j-1)}} \frac{\theta_0 + \lfloor \theta_0^r \rfloor\alpha + j}{\theta_0 + j} \prod_{s = 1}^{\lfloor \theta_0^r \rfloor} \frac{\theta_0 + s \alpha}{\theta_0 + s\alpha + j}.
\eal
Fix $\epsilon > 0$. By Chebyshev's inequality, 
\bal
&P( |X_{m2}(\alpha, \theta_0, \beta, \theta)| > \epsilon) \leq \frac{2}{\epsilon \sqrt{\theta} \tilde{f}} \sum_{k = \lfloor \theta_0^r \rfloor + 1}^\infty \sum_{\bm{m} \in M_{m,L}} \E\left[ \prod_{\ell=1}^L W_{\ell k}^{m_\ell} \right]  \\
&= \frac{2}{\epsilon \sqrt{\theta}\tilde{f}} \sum_{j=1}^m A_j(\beta, m, L) \frac{\prod_{s=0}^{j-1} (\theta + s\beta)}{\beta^j} \sum_{k= \lfloor \theta_0^r \rfloor + 1}^\infty \E[V_k^j] \\
&= \frac{2}{\epsilon \sqrt{\theta}\tilde{f}} \sum_{j=1}^m A_j(\beta, m, L) \frac{ \prod_{s=0}^{j-1} (\theta + s\beta)}{\beta^j} \frac{(1 - \alpha)_{(j-1)}}{(1 +\theta_0)_{(j-1)}} \frac{\theta_0 + \lfloor \theta_0^r \rfloor\alpha + j}{\theta_0 + j} \prod_{s = 1}^{\lfloor \theta_0^r \rfloor} \frac{\theta_0 + s \alpha}{\theta_0 + s\alpha + j} \\
&\sim \frac{2}{\epsilon c \sqrt{\theta} \tilde{f}} \sum_{j=1}^m A_j(\beta, m, L) \frac{(1 - \alpha)_{(j-1)}}{\beta^j c^{j-1}} \cdot \theta_0 \cdot \alpha \theta_0^{r-1} \cdot \theta_0^{-\frac{j(r-1)}{\alpha}} \\
&< \frac{2\alpha}{\epsilon c \sqrt{\theta}} \theta_0^{r - \frac{r-1}{\alpha}} \\
&\to 0
\eal
as $\theta_0, \theta \to \infty$ such that $\frac{\theta_0}{\theta} \to c$, using the fact that $r - \frac{r-1}{\alpha} < 0$ for $r = \lfloor \frac{1}{1-\alpha} \rfloor + 1$ and $\alpha \in [0,1)$. 

Thus it suffices to show that the characteristic function of the truncated sum, $X_{m,L,1}(\alpha, \theta_0, \beta, \theta)$, converges to the characteristic function of $N(0, \sigma_{X,m,L}^2)$. To this end, we first show that the conditional characteristic function of $X_{m,L,1}(\alpha, \theta_0, \beta, \theta)$ given $\bm{V}, \gamma_{1/\beta}$ converges to the characteristic function of $N(0, \sigma_{X,m,L}^2)$. First we identify the conditional variance $\sigma_{X,m,L}^2$. By direct computation and using Lemma \ref{extendedincrementsmean}, 
\bal
&\Var\left( X_{m,L,1}(\alpha, \theta_0, \beta, \theta) \middle\vert \bm{V}, \gamma_{1/\beta} \right) \\
&= \frac{1}{\theta \tilde{f}^2} \sum_{k=1}^{\lfloor \theta_0^r \rfloor} \E\left[ \left( \sum_{\bm{m} \in M_{m,L}} \prod_{\ell=1}^L W_{\ell k}^{m_\ell} - \E\left[ \sum_{\bm{m} \in M_{m,L}} \prod_{\ell=1}^L W_{\ell k}^{m_\ell} \middle\vert \bm{V}, \gamma_{1/\beta}\right] \right)^2 \middle\vert \bm{V}, \gamma_{1/\beta}\right] \\
&= \frac{1}{\theta \tilde{f}^2} \sum_{k=1}^{\lfloor \theta_0^r \rfloor} \left( \E\left[ \left( \sum_{\bm{m} \in M_{m,L}} \prod_{\ell=1}^L W_{\ell k}^{m_\ell} \right)^2 \middle\vert \bm{V}, \gamma_{1/\beta}\right] - \left( \E\left[ \sum_{\bm{m} \in M_{m,L}} \prod_{\ell=1}^L W_{\ell k}^{m_\ell} \middle\vert \bm{V}, \gamma_{1/\beta}\right]\right)^2 \right) \\
&= \frac{1}{\theta \tilde{f}^2} \sum_{k=1}^{\lfloor \theta_0^r \rfloor} \left( \sum_{\bm{m}_1 \in M_{m,L}} \sum_{\bm{m}_2 \in M_{m,L}} \prod_{\ell=1}^L \sum_{j=1}^{m_{1\ell} + m_{2\ell}} \Cs(m_{1\ell} + m_{2\ell}, j, \beta) \gamma_{1/\beta}^j V_k^j - \left( \sum_{j=1}^m A_j(\beta, m, L) \gamma_{1/\beta}^j V_k^j \right)^2 \right) \\
&= \frac{1}{\theta \tilde{f}^2} \sum_{k=1}^{\lfloor \theta_0^r \rfloor} \left( \sum_{j=1}^{2m} \tilde{A}_j(\beta, m, L) \gamma_{1/\beta}^j V_k^j - \sum_{1 \leq i,j \leq m} A_i(\beta, m, L) A_j(\beta, m, L) \gamma_{1/\beta}^{i+j} V_k^{i+j} \right)
\eal
where the coefficients $\{\tilde{A}_j(\beta, m, L)\}_{1 \leq j \leq 2m}$ are defined as in the lemma statement. Therefore by the law of large numbers for a sequence of iid gamma random variables and Lemma \ref{asconvHomozyg},
\bal
&\Var\left( X_{m,L,1}(\alpha, \theta_0, \beta, \theta) \middle\vert \bm{V}, \gamma_{1/\beta} \right) \\
&= \frac{1}{\theta \tilde{f}^2} \left( \sum_{j=1}^{2m} \tilde{A}_j(\beta, m, L) \frac{\prod_{s=0}^{j-1} (\theta + s\beta)}{\beta^j} \left( \frac{\beta^j \gamma_{1/\beta}^j}{\prod_{s=0}^{j-1} (\theta + s\beta)} \right) \frac{(1 - \alpha)_{(j-1)}}{\theta_0^{j-1}} \left( \frac{\theta_0^{j-1}}{(1 - \alpha)_{(j-1)}} \sum_{k=1}^{\lfloor \theta_0^r \rfloor} V_k^j \right) \right. \\
& \qquad  \left. - \sum_{1 \leq i,j \leq m} A_i(\beta, m, L) A_j(\beta, m, L) \frac{\prod_{s=0}^{i+j-1} (\theta + s\beta)}{\beta^{i+j}} \left( \frac{\beta^{i+j} \gamma_{1/\beta}^{i+j}}{\prod_{s=0}^{i+j-1} (\theta + s\beta)} \right) \right. \\
&\qquad \qquad \left. \times \frac{(1 - \alpha)_{(i+j-1)}}{\theta_0^{i+j-1}}  \left( \frac{\theta_0^{i+j-1}}{(1 - \alpha)_{(i+j-1)}} \sum_{k=1}^{\lfloor \theta_0^r \rfloor} V_k^{i+j} \right)  \right) \\
&\xrightarrow{P} \frac{ \sum_{j=1}^{2m} \tilde{A}_j(\beta, m, L) \frac{(1 - \alpha)_{(j-1)}}{\beta^j c^{j-1}} -  \sum_{1 \leq i,j \leq m} A_i(\beta, m, L) A_j(\beta, m, L)  \frac{(1 - \alpha)_{(i+j-1)}}{\beta^{i+j} c^{i+j-1}}}{\tilde{f}^2} =: \sigma_{X,m,L}^2
\eal
as $\theta_0, \theta \to \infty$ such that $\frac{\theta_0}{\theta} \to c$. 

By similar computations, 
\bal
&\sum_{k=1}^{\lfloor \theta_0^r \rfloor} \E\left[\left|\frac{\sum_{\bm{m} \in M_{m,L}} \prod_{\ell = 1}^L W_{\ell k}^{m_\ell} - \E[\sum_{\bm{m} \in M_{m,L}} \prod_{\ell = 1}^L W_{\ell k}^{m_\ell} \mid \bm{V}, \gamma_{1/\beta}]}{\sqrt{\theta} \tilde{f}}\right|^4 \middle\vert \bm{V}, \gamma_{1/\beta} \right] \\
&= \frac{1}{\theta^2 \tilde{f}^4} \sum_{k=1}^{\lfloor \theta_0^r \rfloor} \E\left[\left(\sum_{\bm{m} \in M_{m,L}} \prod_{\ell = 1}^L W_{\ell k}^{m_\ell} - \E\left[\sum_{\bm{m} \in M_{m,L}} \prod_{\ell = 1}^L W_{\ell k}^{m_\ell} \mid \bm{V}, \gamma_{1/\beta}\right]\right)^4 \middle\vert \bm{V}, \gamma_{1/\beta} \right] \\
&= \frac{1}{\theta^2 \tilde{f}^4} \sum_{k=1}^{\lfloor \theta_0^r \rfloor} \left( \E\left[ \left( \sum_{\bm{m} \in M_{m,L}} \prod_{\ell = 1}^L W_{\ell k}^{m_\ell} \right)^4 \middle\vert \bm{V}, \gamma_{1/\beta}\right] \right. \\
&\qquad \qquad \qquad \left. - 4\E\left[ \left( \sum_{\bm{m} \in M_{m,L}} \prod_{\ell = 1}^L W_{\ell k}^{m_\ell} \right)^3 \middle\vert \bm{V}, \gamma_{1/\beta}\right] \E\left[ \left( \sum_{\bm{m} \in M_{m,L}} \prod_{\ell = 1}^L W_{\ell k}^{m_\ell} \right) \middle\vert \bm{V}, \gamma_{1/\beta}\right] \right. \\
& \qquad \qquad \qquad \left. + 6\E\left[ \left( \sum_{\bm{m} \in M_{m,L}} \prod_{\ell = 1}^L W_{\ell k}^{m_\ell} \right)^2 \middle\vert \bm{V}, \gamma_{1/\beta}\right] \left(\E\left[ \left( \sum_{\bm{m} \in M_{m,L}} \prod_{\ell = 1}^L W_{\ell k}^{m_\ell} \right) \middle\vert \bm{V}, \gamma_{1/\beta}\right] \right)^2  \right. \\
& \qquad \qquad \qquad \left. - 3\left(\E\left[ \left( \sum_{\bm{m} \in M_{m,L}} \prod_{\ell = 1}^L W_{\ell k}^{m_\ell} \right) \middle\vert \bm{V}, \gamma_{1/\beta}\right] \right)^4 \right),
\eal
with the following asymptotics 
\bal
&\frac{1}{\theta^2 \tilde{f}^4} \sum_{k=1}^{\lfloor \theta_0^r \rfloor} \E\left[ \left( \sum_{\bm{m} \in M_{m,L}} \prod_{\ell = 1}^L W_{\ell k}^{m_\ell} \right)^4 \middle\vert \bm{V}, \gamma_{1/\beta}\right] \\
&\quad \asymp \frac{1}{\theta \tilde{f}^4} \sum_{j=1}^{4m} \left( \sum_{\bm{m}_1, \bm{m}_2, \bm{m}_3, \bm{m}_4 \in M_{m,L}} \sum_{\bm{j} \in M_{j,L}} \prod_{\ell = 1}^L \Cs(m_{1\ell} + m_{2\ell} + m_{3\ell} + m_{4\ell}, j_\ell, \beta) \right) \frac{(1 - \alpha)_{(j-1)}}{\beta^j c^{j-1}} \\
&\frac{1}{\theta^2 \tilde{f}^4} \sum_{k=1}^{\lfloor \theta_0^r \rfloor} \E\left[ \left( \sum_{\bm{m} \in M_{m,L}} \prod_{\ell = 1}^L W_{\ell k}^{m_\ell} \right)^3 \middle\vert \bm{V}, \gamma_{1/\beta}\right] \E\left[ \left( \sum_{\bm{m} \in M_{m,L}} \prod_{\ell = 1}^L W_{\ell k}^{m_\ell} \right) \middle\vert \bm{V}, \gamma_{1/\beta}\right] \\
&\quad \asymp \frac{1}{\theta \tilde{f}^4} \sum_{\substack{1 \leq i \leq 3m \\ 1 \leq j \leq m}} \left( \sum_{\bm{m}_1, \bm{m}_2, \bm{m}_3 \in M_{m,L}} \sum_{\bm{j} \in M_{j,L}} \prod_{\ell = 1}^L \Cs(m_{1\ell} + m_{2\ell} + m_{3\ell}, j_\ell, \beta) \right) A_j(\beta, m, L) \frac{(1 - \alpha)_{(i+j-1)}}{\beta^{i+j} c^{i+j-1}} \\
&\frac{1}{\theta^2 \tilde{f}^4} \sum_{k=1}^{\lfloor \theta_0^r \rfloor} \E\left[ \left( \sum_{\bm{m} \in M_{m,L}} \prod_{\ell = 1}^L W_{\ell k}^{m_\ell} \right)^2 \middle\vert \bm{V}, \gamma_{1/\beta}\right] \left(\E\left[ \left( \sum_{\bm{m} \in M_{m,L}} \prod_{\ell = 1}^L W_{\ell k}^{m_\ell} \right) \middle\vert \bm{V}, \gamma_{1/\beta}\right] \right)^2 \\
&\quad \asymp \frac{1}{\theta \tilde{f}^4} \sum_{\substack{1 \leq j \leq 2m \\ 1 \leq s,t \leq m}} \tilde{A}_j(\beta, m, L) A_s(\beta, m, L) A_t(\beta, m, L) \frac{(1 - \alpha)_{(j+s+t-1)}}{\beta^{j+s+t} c^{j+s+t-1}} \\ 
&\frac{1}{\theta^2 \tilde{f}^4} \sum_{k=1}^{\lfloor \theta_0^r \rfloor} \left(\E\left[ \left( \sum_{\bm{m} \in M_{m,L}} \prod_{\ell = 1}^L W_{\ell k}^{m_\ell} \right) \middle\vert \bm{V}, \gamma_{1/\beta}\right] \right)^4 \\
&\quad \asymp \frac{1}{\theta \tilde{f}^4} \sum_{1 \leq i, j, s, t \leq m} A_i(\beta, m, L)A_j(\beta, m, L)A_s(\beta, m, L)A_t(\beta, m, L) \frac{(1 - \alpha)_{(i+j+s+t-1)}}{\beta^{i+j+s+t} c^{i+j+s+t-1}} 
\eal
as $\theta_0, \theta \to \infty$ such that $\frac{\theta_0}{\theta} \to c$. Combining the above, it follows that 
\bal
\sum_{k=1}^{\lfloor \theta_0^r \rfloor} \E\left[\left|\frac{W_k^m - \E[W_k^m \mid \bm{V}, \gamma_{1/\beta}]}{\sqrt{\theta} \tilde{f}}\right|^4 \middle\vert \bm{V}, \gamma_{1/\beta} \right] \asymp \frac{1}{\theta} \xrightarrow{P} 0, 
\eal
and in particular, 
\bal
\frac{\sum_{k=1}^{\lfloor \theta_0^r \rfloor} \E\left[\left|\frac{W_k^m - \E[W_k^m \mid \bm{V}, \gamma_{1/\beta}]}{\sqrt{\theta} \tilde{f}}\right|^4 \middle\vert \bm{V}, \gamma_{1/\beta} \right] }{\sqrt{\Var(X_{m1}(\alpha, \theta_0, \beta, \theta) \mid \bm{V}, \gamma_{1/\beta})}^4} \xrightarrow{P} 0
\eal
as $\theta_0, \theta \to \infty$ such that $\frac{\theta_0}{\theta} \to c$. 

Therefore conditional on $\bm{V}$ and $\gamma_{1/\beta}$, the Lyapunov condition with $\delta = 2$ is satisfied. This implies that the conditional Lindeberg-Feller condition is satisfied, and so the conditional characteristic function of $X_{m,L,1}(\alpha, \theta_0, \beta, \theta)$ given $\bm{V}, \gamma_{1/\beta}$ converges to the characteristic function of $N(0, \sigma_{X,m,L}^2)$. Finally by the dominated convergence theorem, the unconditional characteristic function of $X_{m1}(\alpha, \theta_0, \beta, \theta)$ also converges to the characteristic function of $N(0, \sigma_X^2)$, and the result follows. 
\end{proof}

\subsection{Joint convergence in the HPYP with $L$ groups}

Let $Z_k(\alpha, \theta_0, \beta, \theta) = \sqrt{\theta}\left( \frac{\sigma_{\beta, \theta,k}}{\theta} - 1 \right)$ and let
\bal
X_{1,k}(\alpha, \theta_0, \beta, \theta) = \frac{1}{\sqrt{\theta}} \sum_{i=1}^\infty \left( W_{ki}(\alpha, \theta_0, \beta, \theta) - \beta \gamma_{1/\beta} V_i \right) \quad \text{and} \quad T_1(\alpha, \theta_0, \beta, \theta) &= \sqrt{\theta} \left( \frac{\beta \gamma_{1/\beta}}{\theta} - 1 \right),
\eal
so that $Z_k(\alpha, \theta_0, \beta, \theta) = X_{1,k}(\alpha, \theta_0, \beta, \theta) + T_1(\alpha, \theta_0, \beta, \theta)$. 

\begin{lemma}\label{JointCLTXTLGroups}
For $m \geq 2$ and $L \geq 1$, the following multivariate CLTs hold
\bal
\bm{W}_{X,L} &= (X_{1,k}(\alpha, \theta_0, \beta, \theta), X_{m,L}(\alpha, \theta_0, \beta, \theta)) \xrightarrow{D} N(\bm{0}, \bm{\Sigma}_{X,L}), \qquad \text{for all $1 \leq k \leq L$} \\
\bm{W}_{T,L} &= (T_1(\alpha, \theta_0, \beta, \theta), T_{m,L}(\alpha, \theta_0, \beta, \theta)) \xrightarrow{D} N(\bm{0}, \bm{\Sigma}_{T,L}) 
\eal
as $\theta_0, \theta \to \infty$ such that $\frac{\theta_0}{\theta} \to c$, where the covariance matrices are given by
\bal
\bm{\Sigma}_{X,L} &= \begin{pmatrix}
1 - \beta & B_k(\alpha, \beta, c, m, L) \\ 
B_k(\alpha, \beta, c, m, L) & \sigma_{X,m,L}^2 
\end{pmatrix} \\
\bm{\Sigma}_{T,L} &= \begin{pmatrix}
\beta & \beta \frac{\sum_{j=1}^m j A_j(\beta, m, L) \frac{(1 - \alpha)_{(j-1)}}{\beta^j c^{j-1}}}{\sum_{j=1}^m A_j(\beta, m, L) \frac{(1 - \alpha)_{(j-1)}}{\beta^j c^{j-1}}}  \\ 
\beta \frac{\sum_{j=1}^m j A_j(\beta, m, L) \frac{(1 - \alpha)_{(j-1)}}{\beta^j c^{j-1}}}{\sum_{j=1}^m A_j(\beta, m, L) \frac{(1 - \alpha)_{(j-1)}}{\beta^j c^{j-1}}}  & \sigma_{T,m,L}^2 
\end{pmatrix},
\eal
with $\sigma_{X,m,L}^2$ the variance from Lemma \ref{homozygosityLgroupsLevel2.2CLT}, $\sigma_{T,m,L}^2$ the variance from Lemma \ref{homozygosityLgroupsLevel2.1CLT}, and 
\bal
B_k(\alpha, \beta, c, m, L) = \frac{ \sum_{j=1}^m \left(\sum_{\bm{m} \in M_{m,L}} \sum_{\bm{j} \in M_{j, L}} (m_k - \beta j_k) \prod_{\ell = 1}^L \Cs(m_\ell, j_\ell, \beta) \right) \frac{(1 - \alpha)_{(j-1)}}{\beta^j c^{j-1}} }{\sum_{j=1}^m A_j(\beta, m, L) \frac{(1 - \alpha)_{(j-1)}}{\beta^j c^{j-1}}}.
\eal
\end{lemma}

\begin{proof}
We start with $\bm{W}_{X,L}$. For notational simplicity, in what follows we let $W_{ki} = W_{ki}(\alpha, \theta_0, \beta, \theta)$. It suffices to show the joint convergence
\bal
\bm{\hat{W}}_{X,L} = (\hat{X}_{1,k}(\alpha, \theta_0, \beta, \theta), \hat{X}_{m,L}(\alpha, \theta_0, \beta, \theta)) \xrightarrow{D} N(\bm{0}, \bm{\Sigma}_{X,L}),
\eal
where
\bal
\hat{X}_{1,k}(\alpha, \theta_0, \beta, \theta) &= \frac{1}{\sqrt{\theta}} \sum_{i=1}^{\lfloor \theta_0^r \rfloor} (W_{ki} - \beta \gamma_{1/\beta} V_i) \\
\hat{X}_{m,L}(\alpha, \theta_0, \beta, \theta) &= \frac{1}{\sqrt{\theta} \tilde{f}} \sum_{i=1}^{\lfloor \theta_0^r \rfloor} \sum_{\bm{m} \in M_{m,L}} \left( \prod_{\ell = 1}^L W_{\ell i}^{m_\ell} - \E\left[ \prod_{\ell = 1}^L W_{\ell i}^{m_\ell} \middle\vert \bm{V}, \gamma_{1/\beta} \right] \right).
\eal
To this end, we show that the conditional characteristic function of $\bm{\hat{W}}_{X,L}$ converges to the characteristic function of $N(\bm{0}, \bm{\Sigma}_{X,L})$. The proof is similar to the proof of Lemma \ref{homozygosityLgroupsLevel2.2CLT}, and so we will only highlight the main differences. 

Let $\bm{t} = (t_1, t_2) \in \R^2$. For $i \geq 1$, let
\bal
A_i = W_{ki} - \beta \gamma_{1/\beta} V_i \quad \text{and} \quad B_i = \sum_{\bm{m} \in M_{m,L}} \left( \prod_{\ell = 1}^L W_{\ell i}^{m_\ell} - \E\left[ \prod_{\ell = 1}^L W_{\ell i}^{m_\ell} \middle\vert \bm{V}, \gamma_{1/\beta} \right] \right).
\eal
The conditional variance of $\bm{t} \cdot \bm{\hat{W}}_{X,L}$ is 
\bal
\Var\left( \bm{t} \cdot \bm{\hat{W}}_{X,L} \middle\vert \bm{V}, \gamma_{1/\beta} \right) &= \frac{1}{\theta} \sum_{i=1}^{\lfloor \theta_0^r \rfloor} \E\left[ \left( t_1 A_i + \frac{t_2}{\tilde{f}}B_i \right)^2  \middle\vert \bm{V}, \gamma_{1/\beta} \right] \\
&= \frac{1}{\theta} \sum_{i=1}^{\lfloor \theta_0^r \rfloor} \E\left( t_1^2 A_i^2 + \frac{t_2^2}{\tilde{f}^2} B_i^2 + 2\frac{t_1t_2}{\tilde{f}} A_iB_i \middle\vert \bm{V}, \gamma_{1/\beta} \right). 
\eal
By virtue of Lemma \ref{homozygosityLgroupsLevel2.2CLT}, it suffices to compute the cross term $\E\left[ A_iB_i \middle\vert \bm{V}, \gamma_{1/\beta} \right]$. By direct computation, 
\bal
&\E\left[ A_kB_k \middle\vert \bm{V}, \gamma_{1/\beta} \right] = \sum_{\bm{m} \in M_{m,L}} \E\left[ (W_{ki} - \beta \gamma_{1/\beta} V_i)\left( \prod_{\ell = 1}^L W_{\ell i}^{m_\ell} - \E\left[ \prod_{\ell = 1}^L W_{\ell i}^{m_\ell} \middle\vert \bm{V}, \gamma_{1/\beta} \right] \right) \middle\vert \bm{V}, \gamma_{1/\beta} \right] \\
&= \sum_{\bm{m} \in M_{m,L}} \left( \E\left[ W_{ki}^{m_k+1} \prod_{\ell \neq k} W_{\ell i}^{m_\ell} \middle\vert \bm{V}, \gamma_{1/\beta} \right] - \E[W_{ki} \mid \bm{V}, \gamma_{1/\beta}] \E\left[ \prod_{\ell = 1}^L W_{\ell i}^{m_\ell} \middle\vert \bm{V}, \gamma_{1/\beta} \right] \right) \\
&= \sum_{\bm{m} \in M_{m,L}} \left[ \left( \prod_{\ell \neq k} \sum_{j = 1 \wedge m_\ell}^{m_\ell} \Cs(m_\ell, j, \beta) \gamma_{1/\beta}^j V_i^j \right)\left( \sum_{j = 1}^{m_k+1} \Cs(m_k+1, j, \beta) \gamma_{1/\beta}^j V_i^j \right) \right. \\
& \qquad \left. - \beta \gamma_{1/\beta} V_i \prod_{\ell = 1}^L \sum_{j = 1 \wedge m_\ell}^{m_\ell} \Cs(m_\ell, j, \beta) \gamma_{1/\beta}^j V_i^j \right] \\
&= \sum_{\bm{m} \in M_{m,L}} \left( \prod_{\ell \neq k} \sum_{j = 1 \wedge m_\ell}^{m_\ell} \Cs(m_\ell, j, \beta) \gamma_{1/\beta}^j V_i^j \right)\left( \sum_{j=1}^{m_k+1} \Cs(m_k+1, j, \beta) \gamma_{1/\beta}^j V_i^j - \sum_{j=1}^{m_k} \beta \Cs(m_k, j, \beta) \gamma_{1/\beta}^{j+1} V_i^{j+1} \right) .
\eal
By Theorem 2.18 in \cite{Cha05}, the generalized factorial coefficients satisfy the recurrence relation
\bal
\Cs(m+1, j, \beta) = (m - \beta j) \Cs(m,j,\beta) + \beta \Cs(m, j - 1, \beta). 
\eal
Applying this recurrence and using the facts that $\Cs(m,m+1, \beta) = 0$ and $\Cs(m,0,\beta) = 0$ gives
\bal
\sum_{j=1}^{m_k} (m_k - \beta j) \Cs(m_k,j,\beta) \gamma_{1/\beta}^j V_i^j = \sum_{j=1}^{m_k+1} \Cs(m_k+1, j, \beta) \gamma_{1/\beta}^j V_i^j - \sum_{j=1}^{m_k} \beta \Cs(m_k,j,\beta) \gamma_{1/\beta}^{j+1} V_i^{j+1},
\eal
so that 
\bal
\E\left[ A_kB_k \middle\vert \bm{V}, \gamma_{1/\beta} \right] &= \sum_{\bm{m} \in M_{m,L}} \left( \prod_{\ell \neq k} \sum_{j = 1 \wedge m_\ell}^{m_\ell} \Cs(m_\ell, j, \beta) \gamma_{1/\beta}^j V_i^j \right)\left( \sum_{j=1 \wedge m_k}^{m_k} (m_k - \beta j) \Cs(m_k,j,\beta) \gamma_{1/\beta}^j V_i^j \right) \\
&= \sum_{j=1}^m \sum_{\bm{m} \in M_{m,L}} \sum_{\bm{j} \in M_{j,L}} (m_k - \beta j_k) \prod_{\ell = 1}^L \Cs(m_\ell, j_\ell, \beta) \gamma_{1/\beta}^j V_i^j.
\eal
Therefore combining the above gives 
\bal
\Var\left( \bm{t} \cdot \bm{\hat{W}_{X,L}} \middle\vert \bm{V}, \gamma_{1/\beta} \right) &\xrightarrow{P} t_1^2 (1 - \beta) +  t_2^2 \sigma_{X,m,L}^2 + 2t_1t_2\frac{B_k(\alpha, \beta, c, m, L)}{\tilde{f}}
\eal
as $\theta_0, \theta \to \infty$ such that $\frac{\theta_0}{\theta} \to c$, with $B_k(\alpha, \beta, c, m L)$ defined as in the lemma statement. 
The remainder of the proof follows by a similar conditional Lyapunov central limit theorem argument as in the proof of Lemma \ref{homozygosityLgroupsLevel2.2CLT}. 

Next we turn to $\bm{W}_{T,L}$. Let $\bm{s} = (s_1, s_2) \in \R^2$. Then
\bal
\Var\left( \bm{s} \cdot \bm{W}_{T,L} \right) &= \E\left[ \left(s_1 T_1(\alpha, \theta_0, \beta, \theta) + s_2 T_{m,L}(\alpha, \theta_0, \beta, \theta) \right)^2 \right] \\
&= \E\left[ s_1^2 T_1(\alpha, \theta_0, \beta, \theta)^2 + s_2^2 T_{m,L}(\alpha, \theta_0, \beta, \theta)^2 + 2s_1s_2 T_1(\alpha, \theta_0, \beta, \theta)T_{m,L}(\alpha, \theta_0, \beta, \theta) \right] \\
&= s_1^2 \beta + s_2^2 \sigma_T^2 + 2s_1s_2 \E[T_1(\alpha, \theta_0, \beta, \theta)T_{m,L}(\alpha, \theta_0, \beta, \theta)].
\eal
Using the decomposition for $T_{m,L}(\alpha, \theta_0, \beta, \theta)$ from (\ref{Level2.1Decomp}) and the convergence
\bal
\left(\sqrt{\theta} \left( \left(\frac{\beta \gamma_{1/\beta}}{\theta}\right) - 1\right), \sqrt{\theta} \left( \left(\frac{\beta \gamma_{1/\beta}}{\theta}\right)^2 - 1\right), \ldots, \sqrt{\theta} \left( \left(\frac{\beta \gamma_{1/\beta}}{\theta}\right)^m - 1\right) \right) \xrightarrow{D} N(\bm{0}, \bm{\Sigma})
\eal
from (\ref{gammamultivariateCLT}), with covariance matrix given in (\ref{GammaCovMatrix}), we have that
\bal
&\E[T_1(\alpha, \theta_0, \beta, \theta)T_{m,L}(\alpha, \theta_0, \beta, \theta)] \\
&= \E\left[ \sqrt{\theta} \left( \frac{\beta \gamma_{1/\beta}}{\theta} - 1 \right) \frac{\sqrt{\theta/\theta_0}}{\tilde{f}} \sum_{j=1}^m \frac{A_j(\beta, m, L) (1 - \alpha)_{(j-1)}}{\beta^j (\theta_0/\theta)^{j-1}}  \tilde{H}_j(\alpha, \theta_0) \left( \left(\frac{\beta \gamma_{1/\beta}}{\theta}\right)^j - \frac{\prod_{s=0}^{j-1} (\theta + s\beta)}{\theta^j} \right) \right] \\
&\qquad + \E\left[ \frac{1}{\tilde{f}}\sqrt{\theta} \left( \frac{\beta \gamma_{1/\beta}}{\theta} - 1 \right) \sum_{j=1}^m \frac{A_j(\beta, m, L) (1 - \alpha)_{(j-1)}}{\beta^j (\theta_0/\theta)^{j-1}} \sqrt{\theta}  \left( \left(\frac{\beta \gamma_{1/\beta}}{\theta}\right)^j - \frac{\prod_{s=0}^{j-1} (\theta + s\beta)}{\theta^j} \right) \right] \\
&= \frac{1}{\tilde{f}} \sum_{j=1}^m \frac{A_j(\beta, m, L) (1 - \alpha)_{(j-1)}}{\beta^j (\theta_0/\theta)^{j-1}} \E\left[ \sqrt{\theta} \left( \frac{\beta \gamma_{1/\beta}}{\theta} - 1 \right) \sqrt{\theta}  \left( \left(\frac{\beta \gamma_{1/\beta}}{\theta}\right)^j - \frac{\prod_{s=0}^{j-1} (\theta + s\beta)}{\theta^j} \right) \right] \\
&\to \frac{1}{\tilde{f}} \sum_{j=1}^m \frac{A_j(\beta, m, L) (1 - \alpha)_{(j-1)}}{\beta^j c^{j-1}} \cdot \beta j
\eal
as $\theta_0, \theta \to \infty$ such that $\frac{\theta_0}{\theta} \to c$. Therefore 
\bal
\Var\left( \bm{s} \cdot \bm{W}_{T,L} \right) \to s_1^2 \beta + s_2^2 \sigma_{T,m,L}^2 + 2s_1s_2 \beta \frac{\sum_{j=1}^m j \frac{A_j(\beta, m, L) (1 - \alpha)_{(j-1)}}{\beta^j (\theta_0/\theta)^{j-1}}}{\tilde{f}}
\eal
as $\theta_0, \theta \to \infty$ such that $\frac{\theta_0}{\theta} \to c$. The result follows by the Cram\'{e}r-Wold device. 
\end{proof}

Combining Lemmas \ref{homozygosityLgroupsLevel1CLT}, \ref{homozygosityLgroupsLevel2.1CLT}, \ref{homozygosityLgroupsLevel2.2CLT}, and \ref{JointCLTXTLGroups} yields the following multivariate CLT.

\begin{proposition} \label{JointCLTLGroups}
For all $1 \leq k \leq L$, the following multivariate CLT holds 
\bal
(X_{1,k}(\alpha, \theta_0, \beta, \theta), X_m(\alpha, \theta_0, \beta, \theta), T_1(\alpha, \theta_0, \beta, \theta), T_m(\alpha, \theta_0, \beta, \theta), Y_m(\alpha, \theta_0, \beta, \theta)) \xrightarrow{D} N(\bm{0}, \bm{\Sigma})
\eal
as $\theta_0, \theta \to \infty$ such that $\frac{\theta_0}{\theta} \to c$, where the covariance matrix is given by
\bal
\bm{\Sigma} = \begin{pmatrix}
1 - \beta & B_k(\alpha, \beta, c, m, L) & 0 & 0 & 0 \\ 
B_k(\alpha, \beta, c, m, L) & \sigma_{X,m,L}^2 & 0 & 0 & 0 \\
0 & 0 & \beta & \beta \frac{\sum_{j=1}^m j A_j(\beta, m, L) \frac{(1 - \alpha)_{(j-1)}}{\beta^j c^{j-1}}}{\sum_{j=1}^m A_j(\beta, m, L) \frac{(1 - \alpha)_{(j-1)}}{\beta^j c^{j-1}}} & 0 \\ 
0 & 0 & \beta \frac{\sum_{j=1}^m j A_j(\beta, m, L) \frac{(1 - \alpha)_{(j-1)}}{\beta^j c^{j-1}}}{\sum_{j=1}^m A_j(\beta, m, L) \frac{(1 - \alpha)_{(j-1)}}{\beta^j c^{j-1}}} & \sigma_{T,m,L}^2 & 0 \\
0 & 0 & 0 & 0 & \sigma_{Y,L}^2
\end{pmatrix},
\eal
with $\sigma_{X,m,L}^2$ the variance from Lemma \ref{homozygosityLgroupsLevel2.2CLT}, $\sigma_{T,m,L}^2$ the variance from Lemma \ref{homozygosityLgroupsLevel2.1CLT}, $\sigma_{Y,L}^2$ the variance from Lemma \ref{homozygosityLgroupsLevel1CLT}, and $B_k(\alpha, \beta, c, m, L)$ defined in Lemma \ref{JointCLTXTLGroups}. 
\end{proposition}

\begin{proof}
For notational simplicity, in what follows we write $X_{1,k} = X_{1,k}(\alpha, \theta_0, \beta, \theta)$, and similarly for $X_{m,L}$, $T_1$, $T_{m,L}$, and $Y_{m,L}$. Recall the decomposition (\ref{Level2.1Decomp}) and note that the first term converges in probability to $0$ while the second term is only a function of $\gamma_{1/\beta}$. On the other hand, recall that $Y_{m,L}$ is only a function of $\bm{V}$. Noting that $\gamma_{1/\beta}$ and $\bm{V}$ are independent, it follows from Lemma \ref{JointCLTXTLGroups} and Slutsky's theorem that
\bal
(T_1, T_{m,L}, Y_{m,L}) \xrightarrow{D} (N(\bm{0}, \bm{\Sigma}_{T,L}), N(0, \sigma_{Y,L}^2)), 
\eal
as $\theta_0, \theta \to \infty$ such that $\frac{\theta_0}{\theta} \to c$, where $N(\bm{0}, \bm{\Sigma}_T)$ and $N(0, \sigma_Y^2)$ are independent random variables. 

Now let $\bm{t} = (t_1, t_2, t_3, t_4, t_5) \in \R^5$. Next we compute the conditional characteristic function of $\bm{t} \cdot (X_{1,k}, X_{m,L}, T_1, T_{m,L}, Y_{m,L})$, conditioned on $\bm{V}$ and $\gamma_{1/\beta}$. Using the fact that $(X_{1,k}, X_{m,L})$, $(T_1, T_{m,L})$, and $Y_{m,L}$ are conditionally independent given $\bm{V}$ and $\gamma_{1/\beta}$, we have that
\bal
&\E\left[ \exp\left( i \bm{t} \cdot (X_{1,k}, X_{m,L}, T_1, T_{m,L}, Y_{m,L}) \right) \middle\vert \bm{V}, \gamma_{1/\beta}\right] \\
&\qquad = \E\left[ \exp\left( it_1X_{1,k} + it_2X_{m,L} + it_3T_1 + it_4T_{m,L} + it_5Y_{m,L} \right) \middle\vert \bm{V}, \gamma_{1/\beta}\right] \\
&\qquad = \exp\left( it_3T_1 + it_4T_{m,L} + it_5Y_{m,L} \right) \E\left[ \exp\left( it_1X_1 + it_2X_{m,L} \right) \middle\vert \bm{V}, \gamma_{1/\beta}\right].
\eal
Thus the unconditional characteristic function of $\bm{t} \cdot (X_{1,k}, X_{m,L}, T_1, T_{m,L}, Y_{m,L})$ is given by
\bal
&\E\left[ \exp\left( i \bm{t} \cdot (X_1, X_m, T_1, T_m, Y_m) \right)\right] \\
&= \E\left[ \exp\left( it_3T_1 + it_4T_{m,L} + it_5Y_{m,L} \right) \E\left[ \exp\left( it_1X_1 + it_2X_{m,L} \right) \middle\vert \bm{V}, \gamma_{1/\beta}\right] \right] \\
&= \E\left[ \exp\left( it_3T_1 + it_4T_{m,L} + it_5Y_{m,L} \right) \left( \E\left[ \exp\left( it_1X_{1,k} + it_2X_{m,L} \right) \middle\vert \bm{V}, \gamma_{1/\beta}\right] - \E\left[ \exp\left( it_1X_{1,k} + it_2X_{m,L} \right) \right] \right)\right] \\
&\qquad + \E\left[\exp\left( it_3T_1 + it_4T_{m,L} + it_5Y_{m,L} \right) \right] \E\left[ \exp\left( it_1X_{1,k} + it_2X_{m,L} \right) \right].
\eal

The first summand above converges to $0$ by Lemma \ref{JointCLTXTLGroups}, dominated convergence, and the fact that the integrand is upper bounded by $1$. On the other hand, recall that $(T_1, T_{m,L}, Y_{m,L}) \xrightarrow{D} (N(\bm{0}, \bm{\Sigma}_{T,L}), N(0, \sigma_{Y,L}^2))$ where $N(\bm{0}, \bm{\Sigma}_T)$ and $N(0, \sigma_Y^2)$ are independent, and that $(X_{1,k}, X_{m,L}) \xrightarrow{D} N(\bm{0}, \bm{\Sigma}_{X,L})$ from Lemma \ref{JointCLTXTLGroups}. Letting $\bm{Z}_{X,L} \sim N(\bm{0}, \bm{\Sigma}_{X,L})$, $\bm{Z}_{T,L} \sim N(\bm{0}, \bm{\Sigma}_{T,L})$, and $Z_{Y,L} \sim N(0, \sigma_{Y,L}^2)$, the second summand converges as
\bal
&\E\left[ \exp\left( it_3T_1 + it_4T_{m,L} + it_5Y_{m,L} \right)\right] \E\left[ \exp\left( it_1X_{1,k} + it_2X_{m,L} \right) \right] \\
&\qquad \qquad \to \E\left[ \exp(i (t_3, t_4) \cdot \bm{Z}_{T,L} )\right] \E\left[ \exp(i t_5 Z_{Y,L} \right] \E\left[ \exp(i (t_1, t_2) \cdot \bm{Z}_{X,L} ) \right]
\eal 
as $\theta_0, \theta \to \infty$ such that $\frac{\theta_0}{\theta} \to c$. This implies $(X_{1,k}, X_{m,L}, T_1, T_{m,L}, Y_{m,L}) \xrightarrow{D} (\bm{Z}_{X,L}, \bm{Z}_{T,L}, Z_{Y,L})$, where $\bm{Z}_X$, $\bm{Z}_T$, and $Z_Y$ are independent normal random variables. Noting that $(\bm{Z}_{X,L}, \bm{Z}_{T,L}, Z_{Y,L}) \sim N(\bm{0}, \bm{\Sigma})$ completes the proof. 
\end{proof}

\subsection{Proof of the CLT for the homozygosity of the HPYP with $L$ groups}

We are now in a position to prove the CLT for $H_{m,L}(\alpha, \theta_0, \beta, \theta)$. 

\begin{proof}[Proof of Theorem \ref{HPYPCLTLGroups}]
Recall that for $m \geq 2$, 
\bal
\tilde{H}_{m,L}(\alpha, \theta_0, \beta, \theta) &= X_{m,L}(\alpha, \theta_0, \beta, \theta) + T_{m,L}(\alpha, \theta_0, \beta, \theta) + Y_{m,L}(\alpha, \theta_0, \beta, \theta) \\
&\qquad + \sum_{\bm{m} \in M_{m,L}} \sqrt{\theta} \left( \prod_{k=1}^L \left( \frac{\theta}{\sigma_{\beta, \theta, k}} \right)^{m_k} - 1 \right) \frac{\sum_{i=1}^\infty \prod_{k=1}^L W_{ki}^{m_k}}{\theta \tilde{f}}.
\eal
We can rewrite 
\bal
&\sum_{\bm{m} \in M_{m,L}} \sqrt{\theta} \left( \prod_{k=1}^L \left( \frac{\theta}{\sigma_{\beta, \theta, k}} \right)^{m_k} - 1 \right) \frac{\sum_{i=1}^\infty \prod_{k=1}^L W_{ki}^{m_k}}{\theta \tilde{f}} \\
&\quad = \sum_{\bm{m} \in M_{m,L}} C_{\bm{m}} \sqrt{\theta} \left( \prod_{k=1}^L \left( \frac{\theta}{\sigma_{\beta, \theta, k}} \right)^{m_k} - 1 \right) \\
&\qquad \qquad + \sum_{\bm{m} \in M_{m,L}} \sqrt{\theta} \left( \prod_{k=1}^L \left( \frac{\theta}{\sigma_{\beta, \theta, k}} \right)^{m_k} - 1 \right) \left( \frac{\sum_{i=1}^\infty \prod_{k=1}^L W_{ki}^{m_k}}{\theta \tilde{f}} - C_{\bm{m}} \right),
\eal
where 
\bal
C_{\bm{m}} = \frac{\sum_{j=1}^m \sum_{\bm{j} \in M_{j,L}} \prod_{\ell = 1}^L \Cs(m_\ell, j_\ell, \beta) \frac{(1 - \alpha)_{(j-1)}}{\beta^j c^{j-1}}}{\sum_{j=1}^m A_j(\beta, m, L) \frac{(1 - \alpha)_{(j-1)}}{\beta^j c^{j-1}}}. 
\eal
By the law of large numbers, the second summand converges in distribution to $0$ as $\theta_0, \theta \to \infty$ such that $\frac{\theta_0}{\theta} \to c$. On the other hand, 
\bal
Z_k(\alpha, \theta_0, \beta, \theta) = \sqrt{\theta} \left( \sum_{i=1}^\infty \frac{W_{ki}}{\theta} - 1 \right) = \sqrt{\theta}\left( \frac{\sigma_{\beta, \theta, k}}{\theta} - 1 \right) \xrightarrow{D} N(0,1)
\eal
as $\theta \to \infty$, for all $1 \leq k \leq L$. Using the fact that $(\sigma_{\beta, \theta, 1}, \ldots, \sigma_{\beta, \theta, L})$ are conditionally independent given $\bm{V}$ and $\gamma_{1/\beta}$, the following joint convergence holds
\bal
\sqrt{\theta}\left[ (\sigma_{\beta, \theta, 1}, \ldots, \sigma_{\beta, \theta, L})^T - (1, \ldots, 1)^T \right] \xrightarrow{D} N(\bm{0}, \bm{I}_L)
\eal
as $\theta \to \infty$, where $\bm{I}_L$ is the $L \times L$ identity matrix. Define the function $h: \R^L \to \R$ by
\bal
h(x_1, \ldots, x_L) = \sum_{\bm{m} \in M_{m,L}} C_{\bm{m}} \prod_{k=1}^L x_k^{-m_k}.
\eal
By direct computation, 
\bal
\nabla h(1, \ldots, 1)^T \cdot \bm{I}_L \cdot \nabla h(1, \ldots, 1) = \sum_{k=1}^L \left( \sum_{\bm{m} \in M_{m,L}} C_{\bm{m}} m_k \right)^2,
\eal
so by the multivariate delta method, 
\bal
\sum_{\bm{m} \in M_{m,L}} C_{\bm{m}} \sqrt{\theta} \left( \prod_{k=1}^L \left( \frac{\theta}{\sigma_{\beta, \theta, k}} \right)^{m_k} - 1 \right) \xrightarrow{D} N\left( 0, \sum_{k=1}^L \left( \sum_{\bm{m} \in M_{m,L}} C_{\bm{m}} m_k \right)^2 \right)
\eal
as $\theta \to \infty$. Combining the above, we get that 
\bal
\sum_{\bm{m} \in M_{m,L}} \sqrt{\theta} \left( \prod_{k=1}^L \left( \frac{\theta}{\sigma_{\beta, \theta, k}} \right)^{m_k} - 1 \right) \frac{\sum_{i=1}^\infty \prod_{k=1}^L W_{ki}^{m_k}}{\theta \tilde{f}} \xrightarrow{D} W \sim N\left( 0, \sum_{k=1}^L \left( \sum_{\bm{m} \in M_{m,L}} C_{\bm{m}} m_k \right)^2 \right).
\eal

Therefore by Proposition \ref{JointCLTLGroups}, for all $m \geq 2$ we have that 
\bal
\tilde{H}_{m,L}(\alpha, \theta_0, \beta, \theta) \xrightarrow{D} X + T + Y + W \sim N(0, \sigma_{c,m,L}^2)
\eal
as $\theta_0, \theta \to \infty$ such that $\frac{\theta_0}{\theta} \to c$, with variance
\bal
&\sigma_{c,m,L}^2 = \sigma_{X,m,L}^2 + \sigma_{T,m,L}^2 + \sigma_{1,m,L}^2 + \sum_{k=1}^L \left( \sum_{\bm{m} \in M_{m,L}} C_{\bm{m}} m_k \right)^2 \\
&- 2\sum_{k=1}^L \left( \frac{ \sum_{j=1}^m \left[\left(\sum_{\bm{m} \in M_{m,L}} \sum_{\bm{j} \in M_{j, L}} (m_k - \beta j_k) \prod_{\ell = 1}^L \Cs(m_\ell, j_\ell, \beta) \right) + \beta j A_j(\beta, m, L) \right] \frac{(1 - \alpha)_{(j-1)}}{\beta^j c^{j-1}} }{\sum_{j=1}^m A_j(\beta, m, L) \frac{(1 - \alpha)_{(j-1)}}{\beta^j c^{j-1}}} \right)^2. \qedhere
\eal
\end{proof}





\Address

\end{document}